\documentclass{amsart}

\usepackage[english]{babel}
\usepackage[T1]{fontenc}
\usepackage[latin1]{inputenc}

\usepackage{dsfont,bm}
\usepackage{enumerate}
\usepackage[numbers]{natbib}
\usepackage{soul}

\usepackage{a4wide}

\usepackage{amsmath,amscd,amsfonts,amssymb,amsaddr,color}
\usepackage{latexsym, graphicx,rotating, subfig}
\usepackage[pdftex,bookmarks,colorlinks,citecolor=blue]{hyperref}
	
\newtheorem{theo}{Theorem}
\newtheorem{corol}{Corollary}
\newtheorem{prop}{Proposition}
\newtheorem{lemm}{Lemma}

\renewcommand{\d}{\mathrm{d}}
\newcommand{\N}{\mathbb{N}}

\newcommand{\R}{\mathbb{R}}
\renewcommand{\P}{\mathbb{P}}
\newcommand{\E}{\mathbb{E}}
\renewcommand{\L}{\mathbb{L}}

\newcommand{\pen}{\text{\upshape{pen}}}
\newcommand{\argmin}[1]{\underset{#1}{\text{argmin }}}
\newcommand{\argmax}[1]{\underset{#1}{\text{argmax }}}
\newcommand{\Var}{\text{\upshape{Var}}}

\newdimen\AAdi%
\newbox\AAbo%
\def\AArm{\fam0 }
\def\AAk#1#2{\setbox\AAbo=\hbox{#2}\AAdi=\wd\AAbo\kern#1\AAdi{}}%
\def\BBone{{\AArm 1\AAk{-.8}{I}I}}%
\makeatletter
\def\@settitle{\begin{center}%
  \baselineskip14\p@\relax
  \bfseries\LARGE
  \@title
  \end{center}%
}
\makeatother

\date{\today}

\begin{document}
\title{Multivariate Intensity Estimation via Hyperbolic Wavelet Selection}
\author{Nathalie Akakpo}
\email{nathalie.akakpo@upmc.fr}
\address{Laboratoire de Probabilit\'es et Mod\`eles Al\'eatoires (LPMA), UMR 7599\\Universit\'e Pierre et Marie Curie (UPMC), Paris\bigskip\\Centre de Recherches Math\'ematiques (CRM), UMI 3457\\Universit\'e de Montr\'eal (UdeM)}

\maketitle

\begin{abstract}
We propose a new statistical procedure able in some way to overcome the curse of dimensionality without structural assumptions on the function to estimate. It relies on a least-squares type penalized criterion and a new collection of models built from hyperbolic biorthogonal wavelet bases. We study its properties in a unifying intensity estimation framework, where an oracle-type inequality and adaptation to mixed smoothness are shown to hold. Besides, we describe an algorithm for implementing the estimator with a quite reasonable complexity.
\end{abstract}


\tableofcontents
\section{Introduction}

Over the last decades, many wavelet procedures have been developed in various statistical frameworks. Yet, in multivariate settings, most of them are based on isotropic wavelet bases. These indeed have the advantage of being  as easily tractable as their univariate counterparts since each isotropic wavelet is a tensor product of univariate wavelets coming from the same resolution level. Notable counterexamples are~\cite{DonohoCART},~\cite{Neumann} and~\cite{NeumannVonSachs}, or~\cite{ACF2015} and~\cite{ACFMaxiset}. They underline the usefulness of hyperbolic wavelet bases, where coordinatewise varying resolution levels are allowed, so as to recover a wider range of functions, and in particular functions with anisotropic smoothness.

Much attention has also been paid to the so-called curse of dimensionality. A common way to overcome this problem in Statistics is to impose structural assumptions on the function to estimate. In a regression framework, beyond the well-known additive and single-index models,  we may cite the work of~\cite{HorowitzMammen} who propose a spline-based method in an additive model with unknown link function, or the use of ANOVA-like decompositions in~\cite{Ingster} or~\cite{DIT}. Besides, two landmark papers consider a general framework of composite functions, encompassing several classical structural assumptions:~\cite{JLT}  propose a kernel-based procedure in the white noise framework, whereas~\cite{BaraudBirgeComposite} propose a general model selection procedure with a wide scope of applications. Finally, Lepski~\cite{LepskiInd} (see also~\cite{RebellesPointwise,RebellesLp}) consider density estimation with adaptation to a possibly multiplicative structure of the density. In the meanwhile, in the field of Approximation Theory and Numerical Analysis, a renewed interest in function spaces with dominating mixed smoothness has been growing (see for instance~\cite{DTU}), due to their tractability for multivariate integration for instance. Such spaces do not impose any structure, but only that the highest order derivative is a mixed derivative. Surprisingly, in the statistical literature, it seems that only the thresholding-type procedures of~\cite{Neumann} and~\cite{BPenskyPicard} deal with such spaces, either in the white noise framework or in a functional deconvolution model. 

In order to fill this gap, this paper is devoted to a new statistical procedure based on wavelet selection from hyperbolic biorthogonal bases. We underline its universality by studying it in a general intensity estimation framework, encompassing many examples of interest such as density, copula density, Poisson intensity or L\'evy jump intensity estimation. We first define a whole collection of linear subspaces, called models, generated by subsets of the dual hyperbolic basis, and a least-squares type criterion adapted to the norm induced by the primal hyperbolic basis. Then we describe a procedure to choose the best model from the data by using a penalized approach similar to~\cite{BBM}. Our procedure satisfies an oracle-type inequality provided the intensity to estimate is bounded. Besides, it reaches the minimax rate up to a constant factor, or up to a logarithmic factor, over a wide range of spaces with dominating mixed smoothness, and this rate is akin to the one we would obtain in a univariate framework. Notice that, contrary to~\cite{Neumann} or~\cite{BPenskyPicard}, we allow for a greater variety of such spaces (of Sobolev, H\"older or Besov type smoothness) and also for spatially nonhomogeneous smoothness. For that purpose, we prove a key result from nonlinear approximation theory, in the spirit of~\cite{BMApprox}, that may be of interest for other types of model selection procedures (see for instance~\cite{BirgeIHP,BaraudHellinger,BaraudBirgeRho}). Depending on the kind of intensity to estimate, different structural assumptions might make sense, some of which have been considered in~\cite{JLT},~\cite{BaraudBirgeComposite},~\cite{LepskiInd},~\cite{RebellesPointwise,RebellesLp}, but not all. We explain in what respect these structural assumptions fall within the scope of estimation under dominating mixed smoothness. Yet, we emphasize that we do not need to impose any structural assumptions on the target function. Thus in some way our method is adaptive at the same time to many structures. Besides, it can be implemented with a computational complexity linear in the sample size, up to logarithmic factors.

The plan of the paper is as follows. In Section~\ref{sec:framework}, we describe the general intensity estimation framework and several examples of interest. In Section~\ref{sec:onemodel}, we define the so-called pyramidal wavelet models and a least-squares type criterion, and provide a detailed account of estimation on a given model. Section~\ref{sec:selection} is devoted to the choice of an adequate penalty so as to perform data-driven model selection. The optimality of the resulting procedure from the minimax point of view is then discussed in Section~\ref{sec:adaptivity}, under mixed smoothness assumptions. The algorithm for implementing our wavelet procedure and an illustrative example are given in Section~\ref{sec:algorithm}. All proofs are postponed to Section~\ref{sec:proofs}. Let us end with some remark about the notation. Throughout the paper, $C,C_1,\ldots$ will stand for numerical constants, and $C(\theta), C_1(\theta),\ldots$ for positive reals that only depend on some $\theta.$ Their values are allowed to change from line to line.

\section{Framework and examples}\label{sec:framework}

	\subsection{General framework} Let $d\in\N,d\geq 2,$ and $Q=\prod_{k=1}^d \left[a_k,b_k\right]$ be a given hyperrectangle in $\R^d$ equipped with its Borel $\sigma-$algebra $\mathcal B(Q)$ and the Lebesgue measure. We denote by $\L^2(Q)$ the space of square integrable functions on $Q,$ equipped with its usual
norm 
\begin{equation}\label{eq:usualnorm}
\|t\| = \sqrt{\int_Q t^2(x) \d x}
\end{equation}
and  scalar product $\langle . , .\rangle.$ In this article, we are interested in a nonnegative measure on $\mathcal B(Q)$ that admits a bounded density $s$ with respect to the Lebesgue measure, and our aim is to estimate that function $s$ over $Q.$ Given a probability space $(\Omega, \mathcal E,\P),$ we assume that there exists some random measure $M$ defined on $(\Omega, \mathcal E,\P)$, with values in the set of Borel measures on $Q$ such that, for all $A\in\mathcal B(Q),$
\begin{equation}\label{eq:M}
\E\left[M(A)\right] =  \langle \BBone_A, s \rangle.
\end{equation}	
By classical convergence theorems, this condition implies that, for all nonnegative or bounded measurable functions $t,$
\begin{equation}\label{eq:Mgen}
\E\left[\int_Q t \d M\right] = \langle t, s \rangle.
\end{equation}
We assume that we observe some random measure $\widehat M,$ which is close enough to $M$ in a sense to be made precise later. When $M$ can be observed, we set of course $\widehat M=M.$

	\subsection{Examples}\label{sec:examples} Our general framework encompasses several special frameworks of interest, as we shall now show.

\subsubsection{Example 1: density estimation.}\label{sec:density} Given $n\in\N^\star,$ we observe identically distributed random variables $Y_1,\ldots,Y_n$ with common density $s$ with respect to the Lebesgue measure on $Q=\prod_{k=1}^d \left[a_k,b_k\right].$ The observed empirical measure is then given by
$$\widehat M(A)=M(A)=\frac{1}{n}\sum_{i=1}^n \BBone_A(Y_i), \text{ for } A\in\mathcal B(Q),$$
and obviously satisfies~\eqref{eq:M}. 

\subsubsection{Example 2: copula density estimation.}\label{sec:copula} Given $n\in\N^\star,$ we observe independent and identically distributed random variables $X_1,\ldots,X_n$ with values in $\R^d.$ For $i=1,\ldots,n$ and $j=1,\ldots,d,$ the $j$-th coordinate $X_{ij}$ of $X_i$ has continuous distribution function $F_j.$ We recall that, from Sklar's Theorem~\cite{Sklar} (see also~\cite{Nelsen}, for instance), there exists a unique distribution function $C$ on $[0,1]^d$ with uniform marginals such that, for all $(x_1,\ldots,x_d)\in \R^d,$
$$\P(X_{i1}\leq x_1,\ldots, X_{id}\leq x_d) = C(F_1(x_1),\ldots,F_d(x_d)).$$
This function $C$ is called the copula of $X_{i1},\ldots,X_{id}.$ We assume that it admits a density $s$ with respect to the Lebesgue measure on $Q=[0,1]^d.$ Since $C$ is the joint distribution function of 
the $F_j(X_{1j}), j=1,\ldots,d,$ a random measure satisfying~\eqref{eq:M} is given by
$$M(A)=\frac{1}{n}\sum_{i=1}^n \BBone_A\left(F_1(X_{i1}),\ldots,F_d(X_{id})\right), \text{ for } A\in\mathcal B([0,1]^d).$$
As the marginal distributions $F_j$ are usually unknown, we replace them by the empirical distribution functions $\hat F_{nj},$ where 
$$\hat F_{nj}(t)=\frac{1}{n}\sum_{i=1}^n \BBone_{X_{ij}\leq t},$$ 
and define 
$$\widehat M(A)=\frac{1}{n}\sum_{i=1}^n \BBone_A\left(\hat F_{n1}(X_{i1}),\ldots,\hat F_{nd}(X_{id})\right), \text{ for } A\in\mathcal B([0,1]^d).$$

\subsubsection{Example 3: Poisson intensity estimation.}\label{sec:Poisson} Let us denote by $\text{Vol}_d(Q)$ the Lebesgue measure of $Q=\prod_{k=1}^d \left[a_k,b_k\right].$ We observe a Poisson process $N$ on $Q$ whose mean measure has intensity $\text{Vol}_d(Q)s.$  Otherwise said, for all finite family $(A_k)_{1\leq k\leq K}$ of disjoint measurable subsets of $Q,$ $N(A_1),\ldots,N(A_K)$ are independent Poisson random variables with respective parameters $\text{Vol}_d(Q)\int_{A_1} s,\ldots,\text{Vol}_d(Q)\int_{A_K} s.$
Therefore the empirical measure 
$$\widehat M(A)=M(A)=\frac{N(A)}{\text{Vol}_d(Q)}, \text{ for } A\in\mathcal B(Q),$$ 
does satisfy~\eqref{eq:M}. 
We do not assume $s$ to be constant throughout $Q$ so that the Poisson process may be nonhomogeneous.

\subsubsection{Example 4: L\'evy jump intensity estimation (continuous time).}\label{sec:levydensitycont} Let $T$ be a fixed positive real, we observe on $[0,T]$ a L\'evy process $\mathbf X=(X_t)_{t\geq 0}$ with values in $\R^d.$ Otherwise said, $\mathbf X$ is a process starting at $0,$ with stationary and independent increments, and which is continuous in probability with c\`adl\`ag trajectories (see for instance~\cite{Bertoin,Sato,ContTankov}). This process may have jumps, whose sizes are ruled by the so-called jump intensity measure or L\'evy measure. An important example of such process is the compound Poisson process 
$$X_t=\sum_{i=1}^{N_t} \xi_i,t\geq 0,$$
where $(N_t)_{t\geq 0}$ is a univariate homogeneous Poisson process, $(\xi_i)_{i\geq i}$ are i.i.d. with values in $\R^d$ and distribution $\rho$ with no mass at $0,$ and $(N_t)_{t\geq 0}$ and $(\xi_i)_{i\geq i}$ are independent. In this case, $\rho$ is also the L\'evy measure of $\mathbf X.$

Here, we assume that the L\'evy measure admits a density $f$ with respect to the Lebesgue measure on $\R^d\backslash\{0\}.$ Given some compact hyperrectangle $Q=\prod_{k=1}^d \left[a_k,b_k\right] \subset \R^d\backslash\{0\},$ our aim is to estimate the restriction $s$ of $f$ to $Q.$ For that purpose, we use the observed empirical measure 
$$\widehat M(A)=M(A)=\frac{1}{T} \iint\limits_{[0,T]\times A} N(\d t, \d x), \text{ for } A\in\mathcal B(Q).$$ 
A well-known property of L\'evy processes states that the random measure $N$ defined for $B\in \mathcal B\left([0,+\infty)\times \R^d\backslash\{0\}\right)$ by
$$N(B)= \sharp \{t >0/ (t, X_t-X_{t^-}) \in B\}$$
is a Poisson process with mean measure 
\begin{equation*}\label{eq:levymeasure}
\mu(B)=\int\int_B f(x) \d t \d x,
\end{equation*}
so that $M$ satisfies~\eqref{eq:M}.

\subsubsection{Example 5: L\'evy jump intensity estimation (discrete time).}\label{sec:levydensitydisc} The framework is the same as in Example 4, except that $(X_t)_{t\geq 0}$ is not observed. Given some time step $\Delta>0$ and $n\in\N^\star,$ we only have at our disposal the random variables 
$$Y_i=X_{i\Delta}- X_{(i-1)\Delta}, i=1,\ldots n.$$
In order to estimate $s$ on $Q,$ we consider the random measure
$$M(A)=\frac{1}{n\Delta} \iint\limits_{[0,n\Delta]\times A} N(\d t, \d x), \text{ for } A\in\mathcal B(Q),$$
which is unobserved, and replaced for estimation purpose with 
$$\widehat M(A)=\frac{1}{n \Delta }\sum_{i=1}^n \BBone_A(Y_i), \text{ for } A\in\mathcal B(Q).$$

\section{Estimation on a given pyramidal wavelet model}\label{sec:onemodel}

The first step of our estimation procedure relies on the definition of finite dimensional linear subspaces of $\L^2(Q),$  called models, generated by some finite families of biorthogonal wavelets. We only describe here models for $Q=[0,1]^d.$ For a general hyperrectangle $Q$, the adequate models can be deduced by translation and scaling. We then introduce a least-squares type contrast that allows to define an estimator of $s$ within a given wavelet model.

	\subsection{Wavelets on  $\L_2([0,1])$}\label{sec:uniwave_assumptions}
	We shall first introduce a multiresolution analysis and a wavelet basis for $\L_2([0,1])$ satisfying the same general assumptions as in~\cite{Hochmuth} and~\cite{HochmuthMixed}. Concrete examples of wavelet bases satisfying those assumptions may be found in~\cite{CohenInterval} and~\cite{Dahmen} for instance.
In the sequel, we denote by $\kappa$ some positive constant, that only depends on the choice of the bases. We fix the coarsest resolution level at $j_0 \in\N.$ On the one hand, we assume that the scaling spaces 
$$V_j=\text{Vect}\{\phi_\lambda; \lambda \in \Delta_j\} 
\text{ and } 
V^\star_j=\text{Vect}\{\phi^\star_\lambda; \lambda \in \Delta_j\}, j\geq j_0,$$
satisfy the following hypotheses:
\begin{enumerate}[$S.i)$]
\item (Riesz bases) For all $j\geq j_0$,  $\{\phi_\lambda; \lambda \in \Delta_j\}$ are linearly independent functions from $\L_2([0,1])$, so are $\{\phi^\star_\lambda; \lambda \in \Delta_j\}$, and they form Riesz bases of $V_j$ and $V^\star_j$, \textit{i.e.} $\left\|\sum_{\lambda\in\Delta_j} a_\lambda \phi_\lambda\right\| \sim \left(\sum_{\lambda\in\Delta_j} a^2_\lambda \right)^{1/2} \sim \left\|\sum_{\lambda\in\Delta_j} a_\lambda \phi^\star_\lambda\right\|.$ 
\item (Dimension) There exists some nonnegative integer $B$ such that, for all $j\geq j_0$, $\dim(V_j)=\dim(V^\star_j)=\sharp \Delta_j =2^j+B.$
\item (Nesting) For all $j\geq j_0$, $V_j\subset V_{j+1}$ and $V^\star_j\subset V^\star_{j+1}.$
\item (Density) $\overline{\cup_{j\geq j_0} V_j}=\overline{\cup_{j\geq j_0} V^\star_j}=\L_2([0,1])$.
\item (Biorthogonality) Let $j\geq j_0$, for all $\lambda,\mu\in \Delta_j$, $\langle \phi_\lambda, \phi^\star_\mu\rangle =\delta_{\lambda,\mu}.$
\item (Localization) Let $j\geq j_0$, for all $\lambda\in \Delta_j$, $|\text{Supp}(\phi_\lambda)| \sim |\text{Supp}(\phi^\star_\lambda)| \sim 2^{-j}.$
\item (Almost disjoint supports) For all $ j\geq j_0$ and all $\lambda\in \Delta_j$, 
$$\max\left(\sharp\{ \mu\in\Delta_j \text{ s.t. } \text{Supp}(\phi_\lambda) \cap \text{Supp}(\phi_\mu)\neq \varnothing\}, \sharp\{ \mu\in\Delta_j \text{ s.t. } \text{Supp}(\phi^\star_\lambda)\cap \text{Supp}(\phi^\star_\mu)\neq \varnothing\}\right) \leq \kappa.$$
\item (Norms) For all $j\geq j_0$ and all $\lambda\in\Delta_j$, $\|\phi_\lambda\|=\|\phi^\star_\lambda\|=1$ and $\max(\|\phi_\lambda\|_{\infty},\|\phi^\star_\lambda\|_{\infty}) \leq \kappa 2^{j/2}$.
\item (Polynomial reproducibility) The primal scaling spaces are exact of order $N$, \textit{i.e.} for all $j\geq j_0$, $\Pi_{N-1} \subset V_j,$ where $\Pi_{N-1}$ is the set of all polynomial functions with degree $\leq N-1$ over $[0,1].$
\end{enumerate}

On the other hand, the wavelet spaces
$$W_j=\text{Vect}\{\psi_\lambda; \lambda \in \nabla_j\} 
\text{ and }
W^\star_j=\text{Vect}\{\psi^\star_\lambda; \lambda \in \nabla_j\}, j\geq j_0+1,$$
fulfill the following conditions:
\begin{enumerate}[$W.i)$]
\item (Riesz bases) The functions $\{\psi_\lambda;\lambda\in\cup_{j\geq j_0+1} \nabla_j\}$ are linearly independent. Together with the $\{\phi_\lambda;\lambda\in\Delta_{j_0}\}$, they form a Riesz basis for $\L_2([0,1])$. The same holds for the $\psi^\star$ and the $\phi^\star$. 
\item (Orthogonality) For all $j\geq j_0$, $V_{j+1}=V_j\oplus W_{j+1}$ and $V^\star_{j+1}=V^\star_j\oplus W^\star_{j+1}$, with $V_j\perp W^\star_{j+1}$ and $V^\star_j\perp W_{j+1}.$
\item (Biorthogonality) Let $j\geq j_0$, for all $\lambda,\mu\in \nabla_{j+1}$, $\langle \psi_\lambda, \psi^\star_\mu\rangle =\delta_{\lambda,\mu}.$
\item (Localization) Let $j\geq j_0$, for all $\lambda\in \nabla_{j+1}$, $|\text{Supp}(\psi_\lambda)| \sim |\text{Supp}(\psi^\star_\lambda)| \sim 2^{-j}.$
\item (Almost disjoint supports) For all $ j\geq j_0$ and all $\lambda\in \nabla_{j+1}$, 
$$\max\left(\sharp\{ \mu\in\nabla_{j+1} \text{ s.t. } \text{Supp}(\psi_\lambda) \cap \text{Supp}(\psi_\mu)\neq \varnothing\}, \sharp\{ \mu\in\nabla_{j+1} \text{ s.t. } \text{Supp}(\psi^\star_\lambda)\cap \text{Supp}(\psi^\star_\mu)\neq \varnothing\}\right) \leq \kappa.$$
\item (Norms) For all $j\geq j_0$ and all $\lambda\in\nabla_{j+1}$, $\|\psi_\lambda\|=\|\psi^\star_\lambda\|=1$ and $\max(\|\psi_\lambda\|_{\infty},\|\psi^\star_\lambda\|_{\infty}) \leq \kappa 2^{j/2}$.
\item (Fast Wavelet Transform) Let $j\geq j_0$, for all $\lambda\in\nabla_{j+1}$, 
$$\sharp\{\mu\in\Delta_{j+1} | \langle \psi_\lambda ,\phi_\mu\rangle \neq 0 \}\leq \kappa$$
and for all $\mu\in\Delta_{j+1}$
$$ |\langle \psi_\lambda ,\phi_\mu\rangle| \leq \kappa.$$
The same holds for the $\psi^\star_\lambda$ and the $\phi^\star_\lambda.$ 
\end{enumerate}

\bigskip
\textit{Remarks: }
\begin{itemize}
\item These properties imply that any function $f\in\L_2([0,1])$ may be decomposed as 
\begin{equation}\label{eq:decompwavelet1}
f=\sum_{\lambda\in \Delta_{j_0}}\langle f,\phi_\lambda \rangle \phi^\star_\lambda + \sum_{j\geq j_0+1} \sum_{\lambda\in \nabla_j}  \langle f,\psi_\lambda \rangle \psi^\star_\lambda.
\end{equation}
\item Properties $S.ii)$ and $W.ii)$ imply that $\dim(W_{j+1})=2^j.$ 
\item Property $W.vii)$ means in particular that, for each resolution level $j$, any wavelet can be represented as a linear combination of scaling functions from the same resolution level with a number of components bounded independently of the level as well as the amplitude of the coefficients.
\end{itemize}

As is well known, contrary to orthogonal bases, biorthogonal bases allow for both symmetric and smooth wavelets. Besides, properties of dual biorthogonal bases are usually not the same. Usually, in decomposition~\eqref{eq:decompwavelet1}, the analysis wavelets $\phi_\lambda$ and $\psi_\lambda$ are the one with most null moments, whereas the synthesis wavelets $\phi^\star_\lambda$ and $\psi^\star_\lambda$ are the one with greatest smoothness. Yet,  we may sometimes need the following smoothness assumptions on the analysis wavelets (not very restrictive in practice), only to bound residual terms due to the replacement of $M$ with $\widehat M.$ 

\medskip
\noindent
\textit{\textbf{Assumption (L).}    
For all $\lambda\in\Delta_{j_0}$, for all $j\geq j_0$ and all $\mu\in\nabla_{j+1},$ $\phi_\lambda$ and $\psi_\mu$ are Lipschitz functions with Lipschitz norms satisfying $\|\phi_\lambda\|_L \leq \kappa 2^{3j_0/2}$ and $\|\psi_\mu\|_L\leq  \kappa 2^{3j/2}.$}

\medskip
\noindent
We still refer to~\cite{CohenInterval} and~\cite{Dahmen} for examples of wavelet bases satisfying this additional assumption.

	\subsection{Hyperbolic wavelet basis on  $\L_2([0,1]^d)$}
In the sequel, for ease of notation, we set $\N_{j_0} =\{j\in\N, j\geq j_0\},$  $\nabla_{j_0}=\Delta_{j_0}$, $W_{j_0}=V_{j_0}$ and $W^\star_{j_0}=V^\star_{j_0},$ and for $\lambda\in\nabla_{j_0}$, $\psi_\lambda=\phi_\lambda$ and $\psi^\star_\lambda=\phi^\star_\lambda.$ Given a biorthogonal basis of $\L_2([0,1])$ chosen according to~\ref{sec:uniwave_assumptions}, we deduce biorthogonal wavelets of $\L_2([0,1]^d)$ by tensor product. More precisely, for $\boldsymbol{j}=(j_1,\ldots,j_d)\in\N_{j_0}^d$, we set $\bm{\nabla_j}=\nabla_{j_1}\times\ldots\times\nabla_{j_d}$ and for all $\bm{\lambda}=(\lambda_1,\ldots,\lambda_d)\in \bm{\nabla_j}$, we define $\Psi_{\bm{\lambda}}(x_1,\ldots,x_d)=\psi_{\lambda_1}(x_1)\ldots\psi_{\lambda_d}(x_d)$ and  $\Psi^\star_{\bm{\lambda}}(x_1,\ldots,x_d)=\psi^\star_{\lambda_1}(x_1)\ldots\psi^\star_{\lambda_d}(x_d).$ Contrary to most statistical works based on wavelets, we thus allow for tensor products of univariate wavelets coming from different resolution levels $j_1,\ldots,j_d.$ Writing $\Lambda=\cup_{{\bm j}\in \N_{j_0}^d}\bm\nabla_{\bm j},$ the families $\{\Psi_{\bm{\lambda}};\bm{\lambda}\in \Lambda \}$ and  $\{\Psi^\star_{\bm{\lambda}};\bm{\lambda}\in \Lambda\}$ define biorthogonal bases of $\L_2([0,1]^d)$ called biorthogonal hyperbolic bases. Indeed,
$$\L_2([0,1]^d)=\overline{\cup_{j\geq j_0} V_j\otimes \ldots \otimes V_j},$$
and for all $j\geq j_0,$
\begin{align*}
V_j\otimes \ldots \otimes V_j
&=(W_{j_0}\oplus W_{j_0+1}\oplus\ldots\oplus W_j )\otimes\ldots\otimes(W_{j_0}\oplus W_{j_0+1}\oplus\ldots\oplus W_j ) \\
&=\underset{j_0\leq k_1,\ldots,k_d \leq j}\oplus W_{k_1}\otimes\ldots\otimes W_{k_d}.
\end{align*}
In the same way, 
$$\L_2([0,1]^d)=\overline{\cup_{j\geq j_0} \underset{j_0\leq k_1,\ldots,k_d \leq j}\oplus W^\star_{k_1}\otimes\ldots\otimes W^\star_{k_d}}.$$
Besides, they induce on $\L_2([0,1]^d)$ the norms 
\begin{equation}\label{eq:normbiortho}
\|t\|_{\bm\Psi}=\sqrt{\sum_{\bm{\lambda} \in\Lambda} \langle t,\Psi_{\bm{\lambda}}\rangle ^2}\text{ and }\|t\|_{\bm\Psi^\star}=\sqrt{\sum_{\bm{\lambda} \in\Lambda} \langle t,\Psi^\star_{\bm{\lambda}}\rangle ^2},
\end{equation}
which are both equivalent to $\|.\|,$ with equality when the wavelet basis is orthogonal. It should be noticed that the scalar product derived from $\|.\|_{\bm\Psi},$ for instance, is
\begin{equation}\label{eq:psbiortho}
\langle t,u\rangle_{\bm\Psi}=\sum_{\bm{\lambda} \in\Lambda} \langle t,\Psi_{\bm{\lambda}}\rangle\langle u,\Psi_{\bm{\lambda}}\rangle.
\end{equation}
	
	\subsection{Pyramidal models}
A wavelet basis in dimension 1 has a natural pyramidal structure when the wavelets are grouped according to their resolution level. A hyperbolic basis too, provided we define a proper notion of resolution level that takes into account anisotropy: for a wavelet $\Psi_{\bm\lambda}$ or $\Psi^\star_{\bm\lambda}$  with $\bm\lambda\in\bm\nabla_{\bm j},$ we define the global resolution level  as $|\bm j|:=j_1+\ldots+j_d.$  Thus, the supports of all wavelets corresponding to a given global resolution level $\ell\in\N_{dj_0}$ have a volume of roughly $2^{-\ell}$ but exhibit very different shapes. For all $\ell \in\N_{dj_0},$ we define $\bm J_\ell=\{\bm j \in \N_{j_0}^d/|\bm j|=\ell\},$ and $U\bm{\nabla}( \ell)=\cup_{\bm j \in \bm J_\ell} \bm\nabla_{\bm j} $ the index set for $d$-variate wavelets at resolution level $\ell$. 

Given some maximal resolution level $L_\bullet\in\N_{dj_0}$, we define, for all $\ell_1\in\{dj_0+1,\ldots, L_\bullet +1\}$,  the family $\mathcal M^\mathcal P_{\ell_1}$ of all sets $m$ of the form 
$$m=\left(\bigcup_{\ell=dj_0}^{\ell_1-1} U\bm{\nabla}( \ell) \right)
\cup \left( \bigcup_{k=0}^{L_\bullet-\ell_1} m(\ell_1+k)\right),$$ 
where, for all $0\leq k \leq L_\bullet-\ell_1$, $m(\ell_1+k)$ may be any subset of $U\bm{\nabla}(\ell_1+k)$ with $N(\ell_1,k)$ elements. Typically, $N(\ell_1,k)$ will be chosen so as to impose some sparsity: it is expected to be smaller than the total number of wavelets at level $\ell_1+k$ and to decrease when the resolution level increases. An adequate choice of $N(\ell_1,k)$ will be proposed in Proposition~\ref{prop:combinatorial}. Thus, choosing a set in $\mathcal M^\mathcal P_{\ell_1}$ amounts to keep all hyperbolic wavelets at level at most   $\ell_1-1,$ but only a few at deeper levels. We set $\mathcal M^\mathcal P=\cup_{\ell_1=dj_0+1}^{L_\bullet +1} \mathcal M^\mathcal P_{\ell_1}$ and define a pyramidal model as any finite dimensional subspace of the form
$$S^{\star}_m=\text{Vect}\{\Psi^\star_{\bm\lambda}; \bm\lambda \in m \}, \text{ for }m \in \mathcal M^\mathcal P.$$ We denote by $D_m$ the dimension of $S^{\star}_m.$
Setting $m_\bullet=\bigcup_{\ell=dj_0}^{L_\bullet} U\bm{\nabla}(\ell),$ we can see that all pyramidal models are included in $S^{\star}_{m_\bullet}.$

	\subsection{Least-squares type estimator on a pyramidal model}\label{sec:LSonemodel}
Let us fix some model $m\in\mathcal M^\mathcal P.$ If the random measure $M$ is observed, then we can build a least-squares type estimator $\check s_m^\star$ for $s$ with values in $S^\star_m$ and associated with the norm $\|.\|_{\bm\Psi}$ defined by~\eqref{eq:normbiortho}. Indeed, setting
\begin{equation*}
\gamma(t)=\|t\|_{\bm\Psi}^2 - 2 \sum_{\bm\lambda\in\Lambda} \langle t, \Psi_{\bm\lambda} \rangle\check\beta_{\bm\lambda},
\end{equation*}
where $$\check\beta_{\bm\lambda}=\int_Q \Psi_{\bm\lambda} \d M,$$
we deduce from~\eqref{eq:Mgen} that $s$ minimizes over $t\in\L^2(Q)$
\begin{align*}
\|s-t\|^2_{\bm\Psi} - \|s\|^2_{\bm\Psi}
= \|t\|^2_{\bm\Psi}- 2 \sum_{\bm\lambda\in \Lambda} \langle t,\Psi_{\bm\lambda}  \rangle  \langle s,\Psi_{\bm\lambda}  \rangle 
=\E[\gamma(t)],
\end{align*}
so we introduce 
$$\check s^\star_m =\argmin{t\in S^\star_m} \gamma (t).$$
For all sequences of reals $(\alpha_{\bm\lambda})_{\bm\lambda \in m},$
\begin{equation}\label{eq:checksm}
\gamma\left(\sum_{\bm\lambda \in m}  \alpha_{\bm\lambda}  \Psi^\star_{\bm\lambda} \right)
= \sum_{\bm\lambda \in m}\left(\alpha_{\bm\lambda} - \check\beta_{\bm\lambda}\right)^2 - \sum_{\bm\lambda \in m} \check \beta_{\bm\lambda}^2,
\end{equation}
hence 
$$\check s^\star_m =\sum_{\bm\lambda \in m} \check \beta_{\bm\lambda} \Psi^\star_{\bm\lambda}.$$
 
Since we only observe the random measure $\widehat M,$ we consider the pseudo-least-squares contrast
\begin{equation*}
\widehat \gamma(t)=\|t\|_{\bm\Psi}^2 - 2 \sum_{\bm\lambda\in\Lambda} \langle t, \Psi_{\bm\lambda} \rangle{\widehat\beta_{\bm\lambda}},
\end{equation*}
where 
$$\widehat\beta_{\bm\lambda}=\int_Q \Psi_{\bm\lambda} \d \widehat M,$$
and we define the best estimator of $s$ within $S^\star_m$ as
$$\widehat  s^\star_m = \argmin{t\in S^\star_m} \widehat\gamma (t)= \sum_{\bm\lambda \in m} \widehat \beta_{\bm\lambda} \Psi^\star_{\bm\lambda}.$$

	\subsection{Quadratic risk on a pyramidal model}\label{sec:QRonemodel}
Let us introduce the orthogonal projection of $s$ on $S^\star_m$ for the norm $\|.\|_{\bm\Psi},$ that is
$$s^\star_m =\sum_{\bm\lambda \in m}\beta_{\bm\lambda}\Psi^\star_{\bm\lambda},$$
where 
$$\beta_{\bm\lambda}=\langle \Psi_{\bm\lambda},s \rangle.$$
It follows from~\eqref{eq:Mgen} that $\check \beta_{\bm\lambda}$ is an unbiased estimator for $\beta_{\bm\lambda},$ so that  $\check s^\star_m$ is an unbiased estimator for $s^\star_m.$ Thanks to Pythagoras' equality, we recover for $\check s^\star_m$ the usual decomposition 
\begin{equation}\label{eq:riskexact}
\E\left[\|s-\check s^\star_m\|^2_{\bm\Psi}\right] = \|s-s^\star_m\|^2_{\bm\Psi} + \sum_{\bm\lambda\in m} \Var(\check \beta_{\bm\lambda}),
\end{equation}
where the first term is a bias term or approximation error and the second term is a variance term or estimation error. When only $\widehat M$ is observed, combining the triangle inequality, the basic inequality~\eqref{eq:basic} and~\eqref{eq:riskexact} easily provides at least an upper-bound akin to~\eqref{eq:riskexact}, up to a residual term.
\begin{prop}\label{prop:onemodel}
For all $\theta >0,$ 
$$ \E\left[\|s-\hat s^\star_m\|^2_{\bm\Psi}\right]
\leq (1+\theta)\left( \|s- s^\star_m\|^2_{\bm\Psi} +  \sum_{\bm\lambda\in m} \Var(\check \beta_{\bm\lambda})\right) + (1+1/\theta)\E\left[ \|\check s^\star_m-\hat s^\star_m\|^2_{\bm \Psi}\right].$$
When $\widehat M=M,$ $\theta$ can be taken equal to 0 and equality holds.
\end{prop}
\noindent
In all the examples introduced in Section~\ref{sec:examples}, we shall verify that the quadratic risks satisfies, for all $\theta>0,$
\begin{equation}\label{eq:upperonemodel}
\E\left[\|s-\hat s^\star_m\|^2_{\bm\Psi}\right]
\leq c_1 \|s- s^\star_m\|^2_{\bm\Psi} + c_2 \frac{\|s\|_\infty D_m}{\bar n} + r_1(\bar n),
\end{equation}
where $\bar n$ describes the amount of available data, and the residual term $r_1(\bar n)$ does not weigh too much upon the estimation rate.

\subsubsection{Example 1: density estimation (continued).} In this framework, the empirical coefficients are of the form
$$\check  \beta_{\bm\lambda}=\frac1n\sum_{i=1}^n \Psi_{\bm\lambda} (Y_i).$$
As the wavelets are normalized and $s$ is bounded, 
$$ \Var(\check \beta_{\bm\lambda})\leq \frac1n\int_Q\Psi^2_{\bm\lambda} (x) s(x) \d x \leq \frac{\|s\|_\infty}{n},$$
so  
$$ \E\left[\|s-\hat s^\star_m\|^2_{\bm\Psi}\right]
\leq  \|s- s^\star_m\|^2_{\bm\Psi} +  \frac{\|s\|_\infty D_m}{n}.$$
Hence~\eqref{eq:upperonemodel} is satisfied for instance with $\bar n=n,$ $c_1=c_2=1,r_1(n)=0.$

\subsubsection{Example 2: copula density estimation (continued).} In this case,
$$\check  \beta_{\bm\lambda}=\frac1n\sum_{i=1}^n \Psi_{\bm\lambda} \left(F_1(X_{i1}),\ldots,F_d(X_{id})\right),$$
while 
$$\widehat  \beta_{\bm\lambda}=\frac1n\sum_{i=1}^n \Psi_{\bm\lambda} \left(\hat F_{n1}(X_{i1}),\ldots,\hat F_{nd}(X_{id})\right).$$
As in Example 1, $\Var(\check \beta_{\bm\lambda})\leq  \|s\|_\infty/ n.$ Besides we prove in Section~\ref{sec:proofscopula} the following upper-bound for the residual terms.
\begin{prop}\label{prop:copulaone} Under Assumption (L), for all $m\in\mathcal M^{\mathcal P},$
$$ \E[\|\check s^\star_m-\hat s^\star_m\|^2] \leq C(\kappa,d)L^{d-1}_\bullet 2^{4L_\bullet} \log(n)/n.$$
\end{prop}
\noindent
Hence choosing $\bar n=n,$ $2^{4L_\bullet}=\sqrt{n}/\log(n),$  $c_1=c_2=2,$ and $r_1(n)=C(\kappa,d)\log( n)^{d-1}/\sqrt{n}$ yields~\eqref{eq:upperonemodel}.

\subsubsection{Example 3: Poisson intensity estimation (continued).} In this case,
$$\check  \beta_{\bm\lambda}=\frac{1}{\text{Vol}_d(Q)}\int_Q \Psi_{\bm\lambda} (x) N(\d x).$$
From Campbell's formula,
$$ \Var(\check \beta_{\bm\lambda})= \frac{1}{\text{Vol}_d(Q)} \int_Q \Psi^2_{\bm\lambda} (x) s(x) \d x$$
so  
$$ \E\left[\|s-\hat s^\star_m\|^2_{\bm\Psi}\right]
\leq  \|s- s^\star_m\|^2_{\bm\Psi} +  \frac{\|s\|_\infty D_m}{\bar n},$$
with $\bar n=\text{Vol}_d(Q).$

\subsubsection{Example 4: L\'evy jump intensity estimation with continuous time observations (continued).} In this case,
$$\check  \beta_{\bm\lambda}=\frac{1}{T}\iint\limits_{[0,T]\times Q} \Psi_{\bm\lambda} (x) N(\d t, \d x).$$
From Campbell's formula again,
$$ \Var(\check \beta_{\bm\lambda})= \frac{1}{T^2}\iint\limits_{[0,T]\times Q} \Psi^2_{\bm\lambda} (x) s(x)\d t \d x$$
so  
$$ \E\left[\|s-\hat s^\star_m\|^2_{\bm\Psi}\right]
\leq  \|s- s^\star_m\|^2_{\bm\Psi} +  \frac{\|s\|_\infty D_m}{\bar n},$$
with $\bar n=T.$

\subsubsection{Example 5: L\'evy jump intensity estimation with discrete time observations (continued).} In this case, the empirical coefficients and their approximate counterparts are of the form 
$$\check  \beta_{\bm\lambda}=\frac{1}{n\Delta}\iint\limits_{[0,n\Delta]\times Q} \Psi_{\bm\lambda} (x) N(\d t, \d x)\quad\text{and}\quad \widehat  \beta_{\bm\lambda}=\frac{1}{n\Delta}\sum_{i=1}^n \Psi_{\bm\lambda} (X_{i\Delta}-X_{(i-1)\Delta}).$$
We deduce as previously that  $ \Var(\check \beta_{\bm\lambda})\leq  \|s\|_\infty/\bar n$ with $\bar n=n\Delta.$ Besides we can bound the residual term thanks to the following proposition, proved in Section~\ref{sec:proofslevydisc}.

\begin{prop}\label{prop:disclevyone} Under Assumption (L), for all $m\in\mathcal M^{\mathcal P},$ 
$$\E[\|\check s^\star_m-\hat s^\star_m\|^2]
\leq 8\frac{\|s\|_\infty D_m}{n\Delta}+ C(\kappa,d,f,Q) L_\bullet^{d-1} \frac{2^{4L_\bullet}n\Delta^3+ 2^{3L_\bullet}\Delta }{n\Delta}.$$
provided $\Delta$ is small enough.
\end{prop}
\noindent
Assuming $n\Delta^2$ stays bounded while $n\Delta\rightarrow\infty$ as  $n\rightarrow \infty,$ and choosing $2^{4L_\bullet}=n\Delta,$  we deduce that~\eqref{eq:upperonemodel} is satisfied under Assumption (L) with $\bar n=n\Delta, c_1=2, c_2=18,$ and $r_1(\bar n)=C(\kappa,d,f,Q)\log(\bar n)^{d-1}/\bar n.$ Notice that these assumptions on $n$ and $\Delta$ are classical in the so-called framework of high-frequency observations.

\textit{Remark:} Proposition~\ref{prop:disclevyone} extends~\cite{FigueroaDiscrete} to a multivariate model with a complex structure due to the use of hyperbolic wavelets, instead of isotropic ones. Yet, the extension is not so straightforward, so we give a detailed proof in Section~\ref{sec:proofs}.

\section{Wavelet pyramid model selection}\label{sec:selection}

The upper-bound~\eqref{eq:upperonemodel} for the risk on one pyramidal model suggests that a good model should be large enough so that the approximation error is small, and small enough so that the estimation error is small. Without prior knowledge on the function $s$ to estimate, choosing the best pyramidal model is thus impossible. In this section, we describe a data-driven procedure that selects the best pyramidal model from the data, without using any smoothness assumption on $s.$ We provide theoretical results that guarantee the performance of such a procedure. We underline how these properties are linked with the structure of the collection of models.

	\subsection{Penalized pyramid selection} When $M$ is observed, we deduce from~\eqref{eq:riskexact} that
\begin{equation*}
\E\left[\|s-\check s^\star_m\|^2_{\bm\Psi}\right] -  \|s\|^2_{\bm\Psi}  = -\|s^\star_m\|^2_{\bm\Psi} + \sum_{\bm\lambda\in m} \Var(\check \beta_{\bm\lambda})
\end{equation*}
and from~\eqref{eq:checksm} that $\gamma(\check s^\star_m)=-\|\check s^\star_m\|^2_{\bm\Psi}.$ Following the work of~\cite{BBM}, we introduce a penalty function $\pen : \mathcal M^{\mathcal P} \rightarrow \R^+$ and choose a best pyramidal model from the data defined as 
$$\hat m^{\mathcal P}=\argmin{m\in\mathcal M^{\mathcal P}}{\left(\hat \gamma(\hat s_m^\star) + \pen(m)\right)}.$$
In order to choose the pyramidal model with smallest quadratic risk, the penalty $\pen(m)$ is expected to behave roughly as the estimation error within model $m.$ We provide such a penalty in the following Section. Our final estimator for $s$ is then  
$$\tilde s^{\mathcal P}= \hat s_{\hat m^{\mathcal P}}^\star.$$

	\subsection{Combinatorial complexity and choice of the penalty function}
As widely examplified in~\cite{Massart,BGH} for instance, the choice of an adequate penalty depends on the combinatorial complexity of the collection of models, which is measured through the index 
\begin{equation}\label{eq:combindex}
\max_{dj_0+1 \leq \ell_1 \leq L_\bullet +1 } \frac{\log\left (\sharp \mathcal M^\mathcal P_{\ell_1}\right)}{D(\ell_1)},
\end{equation}
where $D(\ell_1)$ is the common dimension of all pyramidal models in $\mathcal M^\mathcal P_{\ell_1}.$ Ideally, this index should be upper-bounded independently of the sample size for the resulting model selection procedure to reach the optimal estimation rate. The following proposition describes the combinatorial complexity of the collection of pyramidal models.

\begin{prop}\label{prop:combinatorial} Let $M=2+B/2^{j_0-1}.$ For all $\ell_1\in\{dj_0+1,\ldots, L_\bullet +1\}$ and all $k\in\{0,\ldots,L_\bullet-\ell_1\},$ let
\begin{equation}\label{eq:N}
N(\ell_1,k)=\lfloor 2\sharp U\bm{\nabla}(\ell_1+k)(k+2)^{-(d+2)}2^{-k} M^{-d}\rfloor
\end{equation} 
and $D(\ell_1)$ be the common dimension of all models in $\mathcal M^\mathcal P_{\ell_1}.$ There exists positive reals $\kappa_1(d),$ $\kappa_2(j_0,B,d)$ and $\kappa_3(j_0,B,d)$ such that
\begin{equation*}\label{eq:cardinalpart} 
\kappa_1(d) (\ell_1-dj_0+d-2)^{d-1}2^{\ell_1} \leq D(\ell_1) \leq   \kappa_2(j_0,B,d)(\ell_1-dj_0+d-2)^{d-1}2^{\ell_1}
\end{equation*} 
and 
\begin{equation*}\label{eq:cardinalsubfamily} 
\log\left (\sharp \mathcal M^\mathcal P_{\ell_1}\right) \leq \kappa_3(j_0,B,d)  D(\ell_1).
\end{equation*} 
\end{prop}
\noindent
We remind that $B$ is defined in Section~\ref{sec:uniwave_assumptions}(Assumption $S.ii)$). Possible values for $\kappa_1,\kappa_2$ and $\kappa_3$ are given in the proof, which is postponed to Section~\ref{sec:proofcombinatorial}. In the same way, we could prove a matching lower-bound for $\log\left (\sharp \mathcal M^\mathcal P_{\ell_1}\right)$ for large enough $\ell_1,$ so that the whole family $\mathcal M^\mathcal P$ contains of order of $L_\bullet^{d-1} 2^{L_\bullet}$ models. Typically, we will choose $L_\bullet$ such that $2^{L_\bullet}$ is a power of the sample size $\bar n.$ So while $\mathcal M^\mathcal P$ contains at least  an exponential number of models, the number of models per dimension is moderate enough so that the combinatorial index~\eqref{eq:combindex} bounded.

From now on, we assume that~\eqref{eq:N} is satisfied, as well as the following hypotheses. For all subfamily $\mathcal T$ of $S^\star_{m_\bullet},$ let 
$$\mathcal Z(\mathcal T)=\sup_{t\in\mathcal T} \left(\int_Q \sum_{\lambda \in m_\bullet} \langle t, \Psi_{\bm\lambda} \rangle \Psi_{\bm \lambda} \d M   -\langle s,t\rangle_{\bm \Psi}\right).$$

\medskip
\noindent
\textit{\textbf{Assumption (Conc).} There exist positive reals $\bar n,\kappa'_1,\kappa'_2,\kappa'_3$ such that, for all countable subfamily $\mathcal T$ of $\left\{t\in S^\star_{m_\bullet} | \| t\|_{\bm\Psi}=1\right\}$ satisfying
$$\sup_{t\in\mathcal T}\left\|\sum_{\lambda \in m_\bullet} \langle t, \Psi_{\bm\lambda} \rangle \Psi_{\bm \lambda}\right \|_\infty  \leq B(\mathcal T)$$
for some positive constant $B(\mathcal T),$ we have, for all $x>0,$
$$\P\left(\mathcal Z(\mathcal T) \geq \kappa'_1 \E\left[\mathcal Z(\mathcal T)\right] + \sqrt{\kappa'_2 \|s\|_\infty \frac{x}{\bar n}} + \kappa'_3 B(\mathcal T) \frac{x}{\bar n}\right) \leq \exp(-x).$$}

\medskip
\noindent
\textit{\textbf{Assumption (Var).} There exist a nonnegative constant $\kappa'_4$ and a collection of estimators $(\hat \sigma^2_{\bm \lambda})_{{\bm \lambda}\in m_\bullet}$ such that, for all ${\bm \lambda} \in m_\bullet$,
$$\E\left[\hat \sigma^2_{\bm \lambda}\right] \leq \kappa'_4 \max(\|s\|_\infty,1).$$ 
Besides there exist a nonnegative constant $\kappa'_5,$ a nonnegative function $w$ such that $w(\bar n)/\bar n \xrightarrow[\bar n \to \infty] {} 0$, and a measurable event $\Omega_\sigma$ on which, for all ${\bm \lambda} \in m_\bullet,$
$$\Var(\check \beta_{\bm \lambda}) \leq \kappa'_5 \frac{\max\{\hat \sigma^2_{\bm \lambda},1\}}{\bar n}$$
and such that
$$p_\sigma:=\P(\Omega_{\sigma}^c)\leq \frac{w(\bar n)}{\bar n}.$$}

\medskip
\noindent
\textit{\textbf{Assumption (Rem).} For the same function $w$ as in Assumption \textbf{(Var)} and some nonnegative constant  $\kappa'_6,$ 
$$\E\left[\|\check s^\star_{m} - \hat s^\star_{m}\|^2_{\bm \Psi}\right] \leq \kappa'_6\frac{\|s\|_\infty D_m}{\bar n} + \frac{w(\bar n)}{\bar n}, \text{ for all } {m\subset  m_\bullet},$$
and 
$$\max\left\{\frac{1}{\bar n (\log(\bar n)/d)^{(d+1)/2}} \sqrt{\E\left[\|\check s^\star_{m_\bullet} - \hat s^\star_{m_\bullet}\|^4_{\bm \Psi}\right]},\sqrt{p_\sigma\E\left[\|\check s^\star_{m_\bullet} - \hat s^\star_{m_\bullet}\|^4_{\bm \Psi}\right]} \right\}\leq  \frac{w(\bar n)}{\bar n}.$$
}

\noindent 
Assumption \textbf{(Conc)} describes how the random measure $M$ concentrates around the measure to estimate. Assumption \textbf{(Var)} ensures that we can estimate the variance terms $\E\left[\|s^\star_m - \check s^\star_m \|^2_{\bm \Psi}\right]$ over each $m\in\mathcal M^{\mathcal P}.$ Last, Assumption \textbf{(Rem)} describes how close $\widehat M$ is to $M.$

\begin{theo}\label{theo:choosepen} Assume that~\eqref{eq:N}, Assumptions \textbf{(Conc)}, \textbf{(Var)},\textbf{(Rem)} are satisfied, and that $\max(\|s\|_\infty,1) \leq \bar R.$ Choose $L_\bullet$ such that
$$2^{L_\bullet}\leq \frac{\bar n}{\left((\log \bar n)/d\right)^{2d}}$$
and a penalty of the form  
$$\pen(m)= \sum_{\lambda\in m}\frac{c_1\hat \sigma^2_{\bm\lambda}+c_2\bar R}{\bar n}, m\in\mathcal M^\mathcal P.$$
If $c_1,c_2$ are positive and large enough, then
$$\E\left[\|s-\tilde s^{\mathcal P}\|^2_{\bm \Psi} \right] \leq C_1 \min_{m\in\mathcal M^{\mathcal P}}\left (\|s-s^\star_m\|^2_{\bm \Psi}+  \frac{\bar R D_m}{\bar n}\right) +C_2 \frac{\max\left\{\|s\|^2_{\bm \Psi},\|s\|_\infty,1\right\}}{\bar n}\left(1+\left(\log (\bar n)/d\right)^{-3(d+1)/2} +w(\bar n)\right)$$
where $C_1$ may depend on $\kappa'_1,\kappa'_2,\kappa'_4,\kappa'_5,\kappa'_6,c_1,c_2$ and $C_2$ may depend $\kappa'_1,\kappa'_2,\kappa'_3,\kappa'_7,j_0,d.$
\end{theo}
\noindent
In practice, the penalty constants $c_1$ and $c_2$ are calibrated by simulation study. We may also replace $\bar R$ in the penalty by $\max\{\|\hat s^\star_{m_\bullet}\|_\infty,1\},$ and extend Theorem~\ref{theo:choosepen} to a random $\bar R$ by using arguments similar to~\cite{AkakpoLacour}.

\subsection{Back to the examples}\label{sec:backtoex} First, two general remarks are in order. For $t\in S^\star_{m_\bullet},$ let $f_t= \sum_{\lambda \in m_\bullet} \langle t, \Psi_{\bm\lambda} \rangle \Psi_{\bm \lambda},$ then $\|f_t\|= \|t\|_{\bm \Psi}$ and by~\eqref{eq:Mgen}, for all countable subfamily $\mathcal T$ of $S^\star_{m_\bullet},$
$$\mathcal Z(\mathcal T)=\sup_{t\in\mathcal T} \left(\int_Q f_t \d M   -\E\left[\int_Q f_t \d M \right]\right).$$
So Assumption \textbf{(Conc)} usually proceeds from a Talagrand type concentration inequality. Besides, we have seen in Section~\ref{sec:QRonemodel} that in general
$$\max_{\bm\lambda \in m_{\bullet}} \Var(\check\beta_{\bm\lambda})\leq \frac{\|s\|_\infty}{\bar n}.$$ Thus, whenever some upper-bound $R_\infty$ for $\|s\|_\infty$ is known, Assumption \textbf{(Var)} is satisfied with $\hat \sigma^2_{\bm \lambda} = R_\infty$ for all $\bm \lambda\in m_\bullet,$ $\Omega_\sigma=\Omega,$ $w(\bar n)=0,\kappa'_4=\kappa'_5=1.$ One may also estimate each variance term: this is what we propose in the following results, proved in Section~\ref{sec:proofcorollaries}.

\begin{corol}\label{corol:density} In the density estimation framework (see~\ref{sec:density}), let $\bar R\geq \max(\|s\|_\infty,1),$ $2^{L_\bullet}=n{\left((\log n)/d\right)^{-2d}},$
$$\hat \sigma^2_{\bm \lambda}=\frac{1}{n(n-1)}\sum_{i=2}^n\sum_{j=1}^{i-1}\left(\Psi_{\bm\lambda}(Y_i)-\Psi_{\bm\lambda}(Y_j)\right)^2 \text{ for all } \bm\lambda \in m_{\bullet},$$
and 
$$\pen(m)= \sum_{\lambda\in m}\frac{c_1\hat \sigma^2_{\bm\lambda}+c_2\bar R}{n}, \text{ for all }m\in\mathcal M^\mathcal P.$$
If $c_1,c_2$ are positive and large enough, then
$$\E\left[\|s-\tilde s^{\mathcal P}\|^2_{\bm \Psi} \right] \leq C_1 \min_{m\in\mathcal M^{\mathcal P}}\left (\|s-s^\star_m\|^2_{\bm \Psi}+  \frac{\bar R D_m}{n}\right) +C_2 \frac{\max\left\{\|s\|^2_{\bm \Psi},\bar R\right\}}{n}$$
where $C_1$ may depend on $\kappa, d,c_1,c_2$ and $C_2$ may depend $\kappa,j_0,d.$
\end{corol}

\begin{corol}\label{corol:copula} In the  copula density estimation framework (see~\ref{sec:copula}), let $\bar R\geq \max(\|s\|_\infty,1)$ and $2^{L_\bullet}=\min\{n^{1/8}(\log n)^{-1/4},n{\left((\log n)/d\right)^{-2d}}\}.$ For all $\bm\lambda \in m_{\bullet},$ define
$$\hat \sigma^2_{\bm \lambda}=\frac{1}{n(n-1)}\sum_{i=2}^n\sum_{j=1}^{i-1}\left(\Psi_{\bm\lambda}\left(\hat F_{n1}(X_{i1}),\ldots,\hat F_{nd}(X_{id})\right)-\Psi_{\bm\lambda}\left(\hat F_{n1}(X_{j1}),\ldots,\hat F_{nd}(X_{jd})\right)\right)^2,$$
and for all $m\in\mathcal M^\mathcal P,$ let
$$\pen(m)= \sum_{\lambda\in m}\frac{c_1\hat \sigma^2_{\bm\lambda}+c_2\bar R}{n}.$$
Under Assumption (L), and if $c_1,c_2$ are positive and large enough, then
$$\E\left[\|s-\tilde s^{\mathcal P}\|^2_{\bm \Psi} \right] \leq C_1 \min_{m\in\mathcal M^{\mathcal P}}\left (\|s-s^\star_m\|^2_{\bm \Psi}+  \frac{\bar R D_m}{n}\right) +C_2 \max\left\{\|s\|^2_{\bm \Psi},\bar R\right\}\frac{(\log n)^{d-1}}{\sqrt{n}}$$
where $C_1$ may depend on $\kappa, d,c_1,c_2$ and $C_2$ may depend $\kappa,j_0,d.$
\end{corol}

\begin{corol}\label{corol:Poisson} In the Poisson intensity estimation framework (see~\ref{sec:Poisson}), let $\bar R\geq \max(\|s\|_\infty,1),$ $2^{L_\bullet}=\text{Vol}_d(Q){\left((\log \text{Vol}_d(Q))/d\right)^{-2d}},$
$$\hat \sigma^2_{\bm \lambda}=\frac{1}{\text{Vol}_d(Q)}\int_Q \Psi^2_{\bm\lambda}\d N, \text{ for all } \bm\lambda \in m_{\bullet},$$
and 
$$\pen(m)= \sum_{\lambda\in m}\frac{c_1\hat \sigma^2_{\bm\lambda}+c_2\bar R}{\text{Vol}_d(Q)}, \text{ for all }m\in\mathcal M^\mathcal P.$$
If $c_1,c_2$ are positive and large enough, then
$$\E\left[\|s-\tilde s^{\mathcal P}\|^2_{\bm \Psi} \right] \leq C_1 \min_{m\in\mathcal M^{\mathcal P}}\left (\|s-s^\star_m\|^2_{\bm \Psi}+  \frac{\bar R D_m}{\text{Vol}_d(Q)}\right) +C_2 \frac{\max\left\{\|s\|^2_{\bm \Psi},\bar R\right\}}{\text{Vol}_d(Q)}$$
where $C_1$ may depend on $\kappa, d,c_1,c_2$ and $C_2$ may depend $\kappa,j_0,d.$
\end{corol}
\noindent 

\begin{corol}\label{corol:levydensitycont} In the L\'evy jump intensity estimation framework with continuous time observations (see~\ref{sec:levydensitycont}), let $\bar R\geq \max(\|s\|_\infty,1),$ $2^{L_\bullet}=T{\left((\log T)/d\right)^{-2d}},$
$$\hat \sigma^2_{\bm \lambda}=\frac{1}{T} \iint\limits_{[0,T]\times Q} \Psi^2_{\bm\lambda}(x) N(\d t, \d x), \text{ for all } \bm\lambda \in m_{\bullet},$$
and 
$$\pen(m)= \sum_{\lambda\in m}\frac{c_1\hat \sigma^2_{\bm\lambda}+c_2\bar R}{T}, \text{ for all }m\in\mathcal M^\mathcal P.$$
If $c_1,c_2$ are positive and large enough, then
$$\E\left[\|s-\tilde s^{\mathcal P}\|^2_{\bm \Psi} \right] \leq C_1 \min_{m\in\mathcal M^{\mathcal P}}\left (\|s-s^\star_m\|^2_{\bm \Psi}+  \frac{\bar R D_m}{T}\right) +C_2 \frac{\max\left\{\|s\|^2_{\bm \Psi},\bar R\right\}}{T}$$
where $C_1$ may depend on $\kappa, d,c_1,c_2$ and $C_2$ may depend $\kappa,j_0,d.$
\end{corol}

\begin{corol}\label{corol:levydensitydisc} In the L\'evy jump intensity estimation framework with discrete time observations (see~\ref{sec:levydensitydisc}), let $\bar R\geq \max(\|s\|_\infty,1),$ $2^{L_\bullet}=\min\left\{(n\Delta)^{1/4},n\Delta{\left(\log (n\Delta)/d\right)^{-2d}}\right\},$
$$\hat \sigma^2_{\bm \lambda}=\frac{1}{n\Delta} \sum_{i=1}^n \Psi^2_{\bm\lambda}\left(X_{i\Delta}- X_{(i-1)\Delta}\right), \text{ for all } \bm\lambda \in m_{\bullet},$$
and 
$$\pen(m)= \sum_{\lambda\in m}\frac{c_1\hat \sigma^2_{\bm\lambda}+c_2\bar R}{n\Delta}, \text{ for all }m\in\mathcal M^\mathcal P.$$
If  Assumption (L) is satisfied, if $n\Delta^2$ stays bounded while $n\Delta\rightarrow\infty$ as  $n\rightarrow \infty,$  and if $c_1,c_2$ are positive and large enough, then
$$\E\left[\|s-\tilde s^{\mathcal P}\|^2_{\bm \Psi} \right] \leq C_1 \min_{m\in\mathcal M^{\mathcal P}}\left (\|s-s^\star_m\|^2_{\bm \Psi}+  \frac{\bar R D_m}{T}\right) +C_2 \max\left\{\|s\|^2_{\bm \Psi},\bar R\right\}\frac{\log^{d-1}(n\Delta)}{n\Delta}$$
where $C_1$ may depend on $\kappa, d,c_1,c_2$ and $C_2$ may depend $\kappa,j_0,d,Q,f.$
\end{corol}

\noindent 
Corollaries~\ref{corol:Poisson},~\ref{corol:levydensitycont},~\ref{corol:levydensitydisc} extend respectively the works of~\cite{ReynaudPoisson,FigueroaHoudre,UK} to a multivariate framework, with a complex family of models allowing for nonhomogeneous smoothness, and a more refined penalty.

\section{Adaptivity to mixed smoothness}\label{sec:adaptivity}

There remains to compare the performance of our procedure $\tilde s^{\mathcal P}$ to that of other estimators. For that purpose, we derive the estimation rate of $\tilde s^{\mathcal P}$  under smoothness assumptions that induce sparsity on the hyperbolic wavelet coefficients of $s.$ We then compare it to the minimax rate.  
	
	\subsection{Function spaces with dominating mixed smoothness} For $\alpha\in\N^\star$ and $1\leq p\leq \infty,$ the mixed Sobolev space with smoothness $\alpha$ measured in the $\L_p-$norm is defined as
$$SW^\alpha_{p,(d)}=\left\{f\in\L_p([0,1]^d) \Bigg| \|f\|_{SW^\alpha_{p,(d)}}:= \sum_{0\leq r_1,\ldots,r_d\leq \alpha} \left\|\frac{\partial^{r_1+\ldots+r_d}f}{\partial^{r_1}x_1\ldots\partial^{r_d}x_d}\right\|_p <\infty\right\},$$
while the classical Sobolev space is
$$W^\alpha_{p,(d)}=\left\{f\in\L_p([0,1]^d) \Bigg| \|f\|_{W^\alpha_{p,(d)}}:= \sum_{0\leq r_1+\ldots+r_d\leq \alpha} \left\|\frac{\partial^{r_1+\ldots+r_d}f}{\partial^{r_1}x_1\ldots\partial^{r_d}x_d}\right\|_p <\infty\right\}.$$
The former contains functions whose highest order derivative is the mixed derivative $\partial^{d\alpha}f /{\partial^\alpha x_1\ldots\partial^\alpha x_d},$ while the latter contains all derivatives up to global order $d\alpha.$ Both spaces coincide in dimension $d=1,$ and otherwise we have the obvious continuous embeddings 
\begin{equation}\label{eq:embedSobol}
W^{d\alpha}_{p,(d)}  \hookrightarrow SW^\alpha_{p,(d)}  \hookrightarrow  W^\alpha_{p,(d)}.
\end{equation}
H\"older and Besov spaces with mixed dominating smoothness may be defined thanks to mixed differences. For $f:[0,1]\rightarrow \R,x\in[0,1]$ and $h>0,$ 
$$\Delta^0_{h}(f,x) =f(x),\quad\Delta^1_{h}(f,x) =f(x+h)-f(x)$$
and more generally, for $r\in\N^\star,$ the $r$-th order univariate difference operator is 
$$\Delta^r_{h}=\Delta^1_{h}\circ \Delta^{r-1}_{h},$$ so that
\begin{equation}\label{eq:binomdiff}
\Delta^r_h(f,x)=\sum_{k=0}^r \binom{r}{k} (-1)^{r-k}f(x+kh).
\end{equation}
Then for $t>0$ the univariate modulus of continuity of order $r$ in $\L_p$ is defined as
$$w_r(f,t)_p=\sup_{0<h<t} \|\Delta^{r}_{h}(f,.)\|_p.$$
For $f:[0,1]^d\rightarrow \R,\mathbf x=(x_1,\ldots,x_d)\in[0,1]^d,r\in\N^\star$ and $h_\ell>0,$ we denote by $\Delta^r_{h_\ell,\ell}$ the univariate difference operator applied to the $\ell$-th coordinate while keeping the other ones fixed, so that
$$\Delta^r_{h_\ell,\ell}(f,\mathbf x)=\sum_{k=0}^r \binom{r}{k} (-1)^{r-k}f(x_1,\ldots,x_\ell+kh_\ell,\ldots,x_d).$$
For any subset $\mathbf e$ of $\{1,\ldots,d\}$ and $\mathbf h=(h_1,\ldots,h_d)\in(0,+\infty)^d,$ the $r$-th order mixed difference operator is given by
$$\Delta^{r,\mathbf e}_{\mathbf h}:= \prod_{\ell \in \mathbf e} \Delta^r_{h_\ell,\ell}.$$
For $\mathbf t=(t_1,\ldots,t_d)\in(0,+\infty)^d,$ we set $\mathbf {t_e}=(t_\ell)_{\ell \in \mathbf e},$ and define the mixed modulus of continuity 
$$w_r^{\mathbf e}(f,\mathbf {t_e})_p=\sup_{0<h_\ell<t_\ell,\ell \in \mathbf e} \|\Delta^{r,\mathbf e}_{\mathbf h}(f,.)\|_p.$$
For $\alpha>0$ and $0< p\leq \infty,$ the mixed H\"older space $SH^\alpha_{p,(d)}$ is the space of all functions $f:[0,1]^d\rightarrow \R$ such that
$$\|f\|_{SH^\alpha_{p,(d)}}:=\sum_{\mathbf e \subset \{1,\ldots,d\}} \sup_{\mathbf t >0} \prod_{\ell \in \mathbf e}{t_\ell ^{-\alpha}}w_{\lfloor \alpha \rfloor +1}^{\mathbf e}(f,\mathbf {t_e})_p$$
is finite, where by convention the term associated with $\mathbf e=\emptyset$ is $\|f\|_p.$ More generally, for $\alpha>0$ and $0<p,q\leq \infty,$ the mixed Besov space $SB^\alpha_{p,q,(d)}$ is the space of all functions $f:[0,1]^d\rightarrow \R$ 
such that
$$\|f\|_{SB^\alpha_{p,q,(d)}}:=\sum_{\mathbf e \subset \{1,\ldots,d\}} \left(\int_{(0,1)}\ldots\int_{(0,1)}\left( \prod_{\ell \in \mathbf e}{t_\ell ^{-\alpha}} w_{\lfloor \alpha \rfloor +1}^{\mathbf e}(f,\mathbf {t_e})_p\right)^q \prod_{\ell \in \mathbf e} \frac{\d t_\ell}{t_\ell}\right)^{1/q},$$
where the $\L_q$-norm is replaced by a sup-norm in case $q=\infty,$ so that $SB^{\alpha}_{p,\infty,(d)}=SH^{\alpha}_{p,(d)}.$ 
By comparison, the usual Besov space $B^\alpha_{p,q,(d)}$ may be defined as the space of all functions $f\in\L_p([0,1]^d)$ such that
$$\|f\|_{B^\alpha_{p,q,(d)}}:=
\left\{
\begin{array}{ll}
       \|f\|_p +\sum_{\ell=1}^d  \left(\int_{(0,1)} \left({t_\ell ^{-\alpha}} w_{\lfloor \alpha \rfloor +1}^{\{\ell\}}(f,t_\ell)_p\right)^q \frac{\d t_\ell}{t_\ell}\right)^{1/q}& \mbox{if } 0<q < \infty \\
      \|f\|_p +\sum_{\ell=1}^d \sup_{t_\ell>0}{t_\ell ^{-\alpha}} w_{\lfloor \alpha \rfloor +1}^{\{\ell\}}(f,t_\ell)_p & \mbox{if } q=\infty
\end{array}
\right.
$$
is finite. Extending~\eqref{eq:embedSobol}, the recent results of~\cite{NguyenSickelEmbed} confirm that the continuous embeddings
\begin{equation*}
B^{d\alpha}_{p,q,(d)}  \hookrightarrow  SB^\alpha_{p,q,(d)}  \hookrightarrow B^\alpha_{p,q,(d)},
\end{equation*}
hold under fairly general assumptions on $\alpha,p,q,d.$.

On the other hand, given $\alpha>0,0<p<\infty, 0<q \leq \infty,$ we define 
$$N_{\bm\Psi,\alpha,p,q}(f) = 
\left\{
\begin{array}{ll}
       \left(\sum_{\ell\geq dj_0} 2^{q\ell(\alpha+1/2-1/p)} \sum_{\bm j\in \bm J_\ell} \left(\sum_{\bm\lambda\in\bm{\nabla_j}} {|\langle f,\Psi_{\bm\lambda}\rangle|^p}\right)^{q/p}\right)^{1/q} & \mbox{if } 0<q < \infty \\
       \sup_{\ell\geq dj_0} 2^{\ell(\alpha+1/2-1/p)} \sup_{\bm j\in\bm J_\ell} \left(\sum_{\bm\lambda\in\bm{\nabla_j}} {|\langle f,\Psi_{\bm\lambda}\rangle|^p}\right)^{1/p} & \mbox{if } q=\infty
\end{array}
\right.
$$
and $N_{\bm\Psi,\alpha,\infty,q}$ in the same way by replacing the $\ell_p$-norm with a sup-norm. 
Then for $\alpha>0,0<p,q\leq \infty,R>0,$ we denote by $\mathcal {SB}(\alpha,p,q,R)$ the set of all functions $f\in\L_p([0,1]^d)$ such that 
$$N_{\bm\Psi,\alpha,p,q}(f)\leq R.$$
Under appropriate conditions on the smoothness of $\bm\Psi^\star,$ that we will assume to be satisfied in the sequel, the sets $\mathcal {SB}(\alpha,p,q,R)$ may be interpreted as balls with radius $R$ in Besov spaces with dominating mixed smoothness $SB^{\alpha}_{p,q,(d)}$ (see for instance~\cite{SchmeisserTriebel,HochmuthMixed,Heping,DTU}). Mixed Sobolev spaces are not easily characterized in terms of wavelet coefficients, but they satisfy the compact embeddings 
$$ SB^{\alpha}_{p,\min(p,2),(d)} \hookrightarrow SW^{\alpha}_{p,(d)} \hookrightarrow SB^{\alpha}_{p,\max(p,2),(d)}, \text{ for } 1<p<\infty$$
and
$$ SB^{\alpha}_{1,1,(d)} \hookrightarrow SW^{\alpha}_{1,(d)} \hookrightarrow SB^{\alpha}_{1,\infty,(d)}$$
(see~\cite{DTU}, Section 3.3). So, without loss of generality, we shall mostly turn our attention to Besov-H\"older spaces in the sequel.

	\subsection{Link with structural assumptions} The following property collects examples of composite functions with mixed dominating smoothness built from lower dimensional functions with classical Sobolev or Besov smoothness. The proof and upper-bounds for the norms of the composite functions are given in Section~\ref{sec:proofcomposite}. An analogous property for (mixed) Sobolev smoothness instead of (mixed) Besov smoothness can be proved straightforwardly. 

\begin{prop}\label{prop:composite} Let  $\alpha>0$ and  $0<p,q\leq \infty.$
\begin{enumerate}[(i)]
\item If $u_1,\ldots,u_d\in B^{\alpha}_{p,q,(1)},$ then $f(\mathbf x)=\sum_{\ell=1}^d u_\ell(x_\ell) \in SB^{\alpha}_{p,q,(d)}.$
\item Let $\mathfrak P$ be some partition of $\{1,\ldots,d\}.$ If, for all $I \in\mathfrak P,u_I \in B^{\alpha_I}_{p,q,(|I|)},$ then $f(\mathbf x)=\prod_{I\in\mathfrak P} u_I(\mathbf x_I)\in SB^{\bar \alpha}_{p,q,(d)}$ where $\bar\alpha=\min_{I\in\mathfrak P} (\alpha_I/|I|).$
\item Let $\alpha\in\N^\star$ and $p>1,$ if  $g\in W^{d\alpha}_{\infty,(1)}$ and $u_\ell \in W^\alpha_{p,{1}}$ for $\ell=1,\ldots,d,$ then $f(\mathbf x)=g\left(\sum_{\ell=1}^d u_\ell(x_\ell)\right)\in SW^{\alpha}_{p,(d)}.$
\item If $f\in SB^{\alpha}_{p,q,(d)}$  with $\alpha>1$ and $\partial^d f/\partial x_1\ldots\partial x_d\in\L_p([0,1]^d),$ then $\partial^d f/\partial x_1\ldots\partial x_d \in SB^{\alpha-1}_{p,q,(d)}.$
\item  If $f_1$ and $f_2\in SB^{\alpha}_{p,p,(d)}$ where either $1<p\leq \infty$ and $\alpha>1/p,$ or $p=1$ and $\alpha\geq1,$ then the product function $\mathbf x \mapsto f_1(\mathbf x)f_2(\mathbf x)\in SB^{\alpha}_{p,p,(d)}.$  
\end{enumerate}
\end{prop}
\noindent Notice that in $(i)$ (resp. $(ii), (iii)$), the assumptions on the component functions $u_\ell,u_I$ or $g$ are not enough to ensure that $f\in B^{d\alpha}_{p,q,(d)}$ (resp. $B^{d\bar\alpha}_{p,q,(d)}, W^{d\alpha}_{p,(d)}$).

\smallskip
\noindent
\textit{Remark: }We believe that a generalization of $(iii)$ to Besov or fractional Sobolev smoothness holds. Yet such a generalization would require refined arguments from Approximation Theory in the spirit of~\cite{BourdaudSickelComposition, MoussaiComposition}  which are beyond the scope of that paper.

The structural assumption $(ii)$ may be satisfied in the multivariate density estimation framework~\ref{sec:density} whenever $Y_1=(Y_{11},\ldots,Y_{1d})$ can be split into independent sub-groups of coordinates, and has recently been considered in~\cite{LepskiInd,RebellesPointwise,RebellesLp}. Case $(i)$ and its generalization $(iii)$ may not be directly of use in our multivariate intensity framework, but they will allow to draw a comparison with~\cite{HorowitzMammen,BaraudBirgeComposite}. Combining $(iii)$ and $(iv)$ is of interest for copula density estimation~\ref{sec:copula}, having in mind that a wide nonparametric family of copulas are Archimedean copulas (see~\cite{Nelsen}, Chapter 4), which have densities of the form
$$s(x_1,\ldots,x_d)=(\phi^{-1})'(x_1)\ldots (\phi^{-1})'(x_d) \phi^{(d)}\left(\phi^{-1}(x_1)+\ldots+\phi^{-1}(x_d)\right)$$
provided the generator $\phi$ is smooth enough (see for instance~\cite{McNeilNeslehova}). Combining $(iii), (iv),(v)$ may be of interest for L\'evy intensity estimation in~\ref{sec:levydensitycont} or \ref{sec:levydensitydisc}. Indeed, a popular way to build multivariate L\'evy intensities is based on L\'evy copulas studied in~\cite{KallsenTankov} (see also~\cite{ContTankov}, Chapter 5). The resulting L\'evy intensities then have the form  
$$f(x_1,\ldots,x_d)= f_1(x_1)\ldots f_d(x_d) F^{(1,\ldots,1)}(U_1(x_1)+\ldots +U_d(x_d))$$
where $F$ is a so-called L\'evy copula,  $F^{(1,\ldots,1)}=\partial^d {F}/ \partial t_1 \ldots \partial t_d$ and $U_\ell(x_\ell)=\int_{x_\ell}^\infty f_\ell(t)\d t.$ Besides, a common form for $F$ is 
$$F(x)=\phi\left(\phi^{-1}(x_1)+\ldots+\phi^{-1}(x_d)\right)$$
under appropriate smoothness assumptions on $\phi.$ Last, let us emphasize that any linear combination (mixtures for instance) of functions in $SB^{\alpha}_{p,q,(d)}$ inherits the same smoothness. Consequently, mixed dominating smoothness may be thought as a fully nonparametric surrogate for a wide range of structural assumptions.

	\subsection{Approximation qualities and minimax rate}

We provide in Section~\ref{sec:proofapprox} a constructive proof for the following nonlinear approximation result, in the spirit of~\cite{BBM}.
\begin{theo}\label{theo:approx}
Let  $R>0,0<p<\infty, 0<q\leq \infty,\alpha >\max(1/p-1/2,0)$,  and $f\in\L_2([0,1]^d)\cap \mathcal {SB}(\alpha,p,q,R).$ Under~\eqref{eq:N}, for all $\ell_1\in\{dj_0+1,\ldots,L_\bullet +1\}$, there exists some model $m_{\ell_1}(f)\in\mathcal M_{\ell_1}^{\mathcal P}$ and some approximation $A(f,\ell_1)\in S^\star_{m_{\ell_1}(f)}$ for $f$ such that 
\begin{align*}\label{eq:upperboundapprox}
&\|f-A(f,\ell_1)\|^2_{\bm\Psi} \\&\leq C(B,j_0,\alpha,p, d) R^2\left( L_\bullet^{2(d-1)(1/2-1/q)_+} 2^{-2L_\bullet(\alpha-(1/p-1/2)_+)}+\ell_1^{2(d-1)(1/2-1/\max(p,q))}2^{-2\alpha \ell_1} \right).
\end{align*}
\end{theo}
\noindent
\textit{Remark:} When $p\geq 2,$ the same kind of result still holds with all $N(\ell_1,k)=0.$ But Assumption~\eqref{eq:N} is really useful when $p<2,$ the so-called non-homogeneous smoothness case.

\medskip
\noindent
The first term in the upper-bound is a linear approximation error by the highest dimensional model $S^\star_{m_\bullet}$ in the collection. As $D_{m_\bullet}$ is of order $L_\bullet^{d-1}2^{L_\bullet},$ we deduce from~\cite{DTU} (Section 4.3) that this first term is optimal over $SB^\alpha_{p,q},$ at least for $1< p <\infty,1\leq q\leq \infty$ and $\alpha>\max(1/p-1/2,0),$ for instance. The second term in the upper-bound is a nonlinear approximation error of $f$ within the model $S^\star_{m_{\ell_1}(f)},$ with dimension $D_{m_{\ell_1}(f)}$ of order $\ell_1^{d-1}2^{\ell_1}.$ So we deduce from~\cite{DTU} (Theorem 7.6) that this second term, which is of order $D_{m_{\ell_1}(f)}^{-2\alpha}(\log D_{m_{\ell_1}(f)} )^{2(d-1)(\alpha+1/2-1/q)},$  is also optimal up to a constant factor over $SB^\alpha_{p,q},$ at least for $1< p <\infty, p\leq q\leq \infty$ and $\alpha>\max(1/p-1/2,0).$ Notice that, under the classical Besov smoothness assumption $f\in\L_2([0,1]^d)\cap B^{\alpha}_{p,q,(d)},$ the best possible approximation rate for $f$ by $D$-dimensional linear subspaces in the $\L_2$-norm would be of order $D^{-2\alpha/d}.$ Thus with a mixed smoothness of order $\alpha$ in dimension $d,$ we recover the same approximation rate as with a classical smoothness of order $d\alpha$ in dimension $d,$ up to a logarithmic factor. 

Let us define, for $\alpha,p,q,R,R'>0,$
$$\overline{\mathcal {SB}}(\alpha,p,q,R,R')=\{f \in \mathcal {SB}(\alpha,p,q,R)/ \|f\|_\infty \leq R'\}.$$ In the sequel, we use the notation $a\asymp C(\theta) b$ when there exist positive reals $C_1(\theta),C_2(\theta)$ such that $C_1(\theta) b\leq a \leq C_2(\theta) b.$

\begin{corol}\label{corol:upperrate} Assume $L_\bullet$ is large enough, then for all $0<p<\infty,0<q\leq \infty,\alpha>(1/p-1/2)_+,R\geq \bar n^{-1},R'>0,$
$$\sup_{s\in\overline{\mathcal {SB}}(\alpha,p,q,R,R')} \E_s \left[\|s-\tilde s^{\mathcal P}\|^2_{\bm\Psi}\right] \leq C(B,d,\alpha,p,R') \left( \left(\log(\bar n R^2)\right)^{(d-1)(\alpha+1/2-1/\max(p,q))} R\bar n^{-\alpha}\right)^{2/(1+2\alpha)}.$$
\end{corol}

\begin{proof} In order to minimize approximately the upper-bound, we choose $\ell_1$ such that
$$ \ell_1^{2(d-1)(1/2-1/\max(p,q))} 2^{-2\alpha \ell_1} R^2  \asymp C(\alpha,p,q,d)  \ell_1^{d-1} 2^{\ell_1}/\bar n,$$
that is for instance
 $$2^{\ell_1} \asymp C(\alpha,p,q,d) \left( \left(\log(\bar n R^2)\right)^{-2(d-1)/\max(p,q)}  (\bar n R^2)\right)^{1/(1+2\alpha)},$$
 which yields the announced upper-bound.
\end{proof}

\noindent
Remember that a similar result holds when replacing the $\bm\Psi$-norm by the equivalent $\L_2$-norm. Though unusual, the upper-bound in Corollary~\ref{corol:upperrate} is indeed related to the minimax rate. 
\begin{prop}  In the density estimation framework, assume $R^2\geq  n^{-1},R'>0,p>0, 0<q\leq \infty,$  and either $\alpha>(1/p-1/2)_+$ and $q\geq 2$ or $\alpha>(1/p-1/2)_++1/\min(p,q,2) - 1/\min(p,2),$ then
$$\inf_{\hat s \text{ estimator of } s }\sup_{s\in\overline{\mathcal {SB}}(\alpha,p,q,R,R')} \E_s \left[\|s-\hat s\|^2\right] \asymp C(\alpha,p,q,d)\left( \left(\log(nR^2)\right)^{(d-1)(\alpha+1/2-1/q)} R n^{-\alpha}\right)^{2/(1+2\alpha)}.$$
\end{prop}

\begin{proof}
One may derive from~\cite{DTU} (Theorem 6.20), ~\cite{Dung} (proof of Theorem 1) and the link between entropy number and Kolmogorov entropy that the Kolmogorov $\epsilon$-entropy of $\mathcal {SB}(\alpha,p,q,R)$ is
$$H_\epsilon(\alpha,p,q,R)=(R/\epsilon)^{1/\alpha} \left(\log(R/\epsilon)\right)^{(d-1)(\alpha+1/2-1/q)/\alpha}.$$
According to~\cite{YangBarron} (Proposition 1), in the density estimation framework, the minimax risk over $\overline{\mathcal {SB}}(\alpha,p,q,R,R')$ is of order $\rho^2_n$ where $\rho^2_n=H_{\rho_n}(\alpha,p,q,R)/n,$ which yields the announced rate. 
\end{proof}

\noindent
Consequently, in the density estimation framework, the penalized pyramid selection procedure is minimax over $\mathcal {SB}(\alpha,p,q,R)$ up to a constant factor if $p\leq q\leq \infty,$ and only up to a logarithmic factor otherwise.

Let us end with some comments about these estimation rates. First, we remind that the minimax rate under the assumption $s\in B^\alpha_{p,q,(d)}$ is of order $n^{-2\alpha/d/(1+2\alpha/d)}.$ Thus, under a mixed smoothness assumption of order $\alpha,$ we recover, up to a logarithmic factor, the same rate as with smoothness of order $\alpha$ in dimension 1, which can only be obtained with smoothness of order $d\alpha$ under a classical smoothness assumption in dimension $d$. Besides, under the multiplicative constraint $(ii)$ of Proposition~\ref{prop:composite}, we recover the same rate as~\cite{RebellesLp}, up to a logarithmic factor. And under the generalized additive constraint $(iii)$ of Proposition~\ref{prop:composite}, we recover the same rate as~\cite{BaraudBirgeComposite} (Section 4.3), up to a logarithmic factor. Regarding Neumann seminal work on estimation under mixed smoothness~\cite{Neumann} (see his Section 3), a first adaptive wavelet thresholding is proved to be optimal up to a logarithmic factor over $SW^{r}_{2,(d)}=SB^{r}_{2,2,(d)},$ and another, nonadaptive one, is proved to be optimal up to a constant over $SB^{r}_{1,\infty,(d)},$ where $r$ is a positive integer. Our procedure thus outperforms~\cite{Neumann} by being at the same time adaptive and minimax optimal up to a constant over these two classes, and many other ones.

\section{Implementing wavelet pyramid selection}\label{sec:algorithm}

We end this paper with a quick overview of practical issues related to wavelet pyramid selection. As we perform selection within a large collection of models, where typically the number of models is exponential in the sample size, we must guarantee that the estimator can still be computed in a reasonable time. Besides, we provide simulation based examples illustrating the interest of this new method. 

	\subsection{Algorithm and computational complexity}\label{sec:twostepalgo} Theorem~\ref{theo:choosepen} supports the choice of an additive penalty of the form 
$$\pen(m)=\sum_{\bm\lambda\in m} \hat v^2_{\bm\lambda},$$
where detailed expressions for $\hat v^2_{\bm\lambda}$  in several statistical frameworks have been given in Section~\ref{sec:backtoex}. 
As $\hat \gamma(\hat s_m^\star)=-\sum_{\bm\lambda\in m} \hat \beta^2_{\bm\lambda},$ the penalized selection procedure amounts to choose
$$\hat m ^{\mathcal P}=\argmax{m\in\mathcal M^{\mathcal P}}{\text{crit}(m)}$$
where 
$$\text{crit}(m)=\sum_{\bm\lambda \in m}(\hat \beta^2_{\bm\lambda}- \hat v_{\bm\lambda}^2).$$
Since each $\hat v^2_{\bm\lambda}$ is roughly an (over)estimate for the variance of $\hat \beta^2_{\bm\lambda}, $  our method, though different from a thresholding procedure, will mainly retain empirical wavelet coefficients $\hat \beta^2_{\bm\lambda}$ which are significantly larger than their variance.

A remarkable thing is that, due to both the structure of the collection of models and of the penalty function, the penalized estimator can be determined without computing all the preliminary estimators $(\hat s^\star_m)_{m\in\mathcal M^{\mathcal P}},$ which makes the computation of $\tilde s^{\mathcal P}$ feasible in practice. Indeed, we can proceed as follows. 

\textit{Step 1.} For each $\ell_1\in\{dj_0+1,\ldots,L_\bullet+1\}$, determine 
$$\hat m_{\ell_1}=\argmax{m\in\mathcal M_{\ell_1}^{\mathcal P}}{\sum_{\bm\lambda \in m}(\hat \beta^2_{\bm\lambda}- \hat v_{\bm\lambda}^2)}.$$
For that purpose, it is enough, for each $k\in\{0,\ldots,L_\bullet-\ell_1\},$ to 
\begin{itemize}
\item compute and sort in decreasing order all the coefficients $(\hat \beta^2_{\bm\lambda}- \hat v_{\bm\lambda}^2)_{\bm\lambda \in U\bm\nabla(\ell_1+k)};$
\item keep the $N(\ell_1,k)$ indices in $U\bm\nabla(\ell_1+k)$ that yield the $N(\ell_1,k)$ greatest such coefficients.
\end{itemize}

\textit{Step 2.} Determine the integer $\hat \ell \in\{dj_0+1,\ldots,L_\bullet+1\}$ such that
$$\hat m_{\hat \ell}=\argmax{dj_0+1\leq \ell_1 \leq L_\bullet+1}{\text{crit}(\hat m_{\ell_1})}.$$

\noindent
The global computational complexity of $\tilde s^{\mathcal P}$ is thus $\mathcal O (\log(L_\bullet) L_\bullet^d 2^{L_\bullet}).$ Typically, we will choose $L_\bullet$ at most of order $\log_2(\bar n)$ so the resulting computational complexity will be at most of order $\mathcal O (\log(\log(\bar n))\log^d(\bar n)\bar n).$

	\subsection{Illustrative examples}

In this section, we study two examples in dimension $d=2$ by using Haar wavelets. 

First, in the density estimation framework, we consider an example where the coordinates of $Y_i=(Y_{i1},Y_{i2})$ are independent conditionally on a $K$-way categorical variable $Z,$ so that the density of $Y_i$ may be written as 
$$s(x_1,x_2)=\sum_{k=1}^K \pi_k s_{1,k}(x_1)s_{2,k}(x_2),$$
where $\mathbf \pi= (\pi_1,\ldots,\pi_K)$ is the probability vector characterizing the distribution of $Z.$ For a compact interval $I,$ and $a,b>0,$ let us denote by $\beta(I;a,b)$ the Beta density with parameters $a,b$ shifted and rescaled to have support $I,$ and by $\mathcal U(I)$ the uniform density on $I.$ In our example, we take
\begin{itemize}
\item $K=4$ and $\pi=\left(3/5,1/10,1/40,11/40\right);$
\item $s_{1,1}=\beta\left([0,3/5];4,4\right)$ and $s_{2,1}=\beta\left([0,2/5];4,4\right);$
\item $s_{1,2}=\beta\left([2/5,1];100,100\right)$ and $s_{2,2}=\beta\left([2/5,1];20,20\right);$
\item $s_{1,3}=\mathcal U\left([0,1]\right)$ and $s_{2,3}=\mathcal U\left([0,1]\right);$
\item $s_{1,4}=\beta\left([3/5,1];8,4\right)$ and $s_{2,4}=\mathcal U\left([2/5,1]\right).$
\end{itemize}	
The resulting mixture density $s$ of $Y_i$ is shown in Figure~\ref{fig:densexample} $(b)$. We choose $L_\bullet=\left[n/((\log n)/2)^2\right]$ and first compute the least-squares estimator $\hat s^\star_\bullet$ of $s$ on the model $V^\star_{L_\bullet/2}\otimes\ldots \otimes V^\star_{L_\bullet/2},$ which provides the estimator $\hat R=\max\{\|\hat s^\star_\bullet\|,1\}$ for $\bar R.$ We then use the penalty 
$$\pen(m)=\sum_{\bm\lambda \in m} \frac{1.5 \hat \sigma^2_{\bm \lambda}+0.5 \hat R}{n}.$$
For a sample with size $n=2000,$ Figure~\ref{fig:densexample} illustrates how the procedure first selects a rough model $\hat m_{\hat\ell}$ (Figure~\ref{fig:densexample} (c))  and then add some details wherever needed (Figure~\ref{fig:densexample} (d)). Summing up the two yields the pyramid selection estimator $\tilde s^\mathcal P$ (Figure~\ref{fig:densexample} (e)). By way of comparison, we also represent in Figure~\ref{fig:densexample} (f) a widely used estimator: the bivariate Gaussian kernel estimator, with the "known support" option, implemented in MATLAB {\tt{ksdensity}} function. We observe that, contrary to the kernel density estimator, the pyramid selection estimator recovers indeed the main three modes, and in particular the sharp peak.

\begin{centering}
\begin{figure}[h]
\includegraphics[scale=0.2]{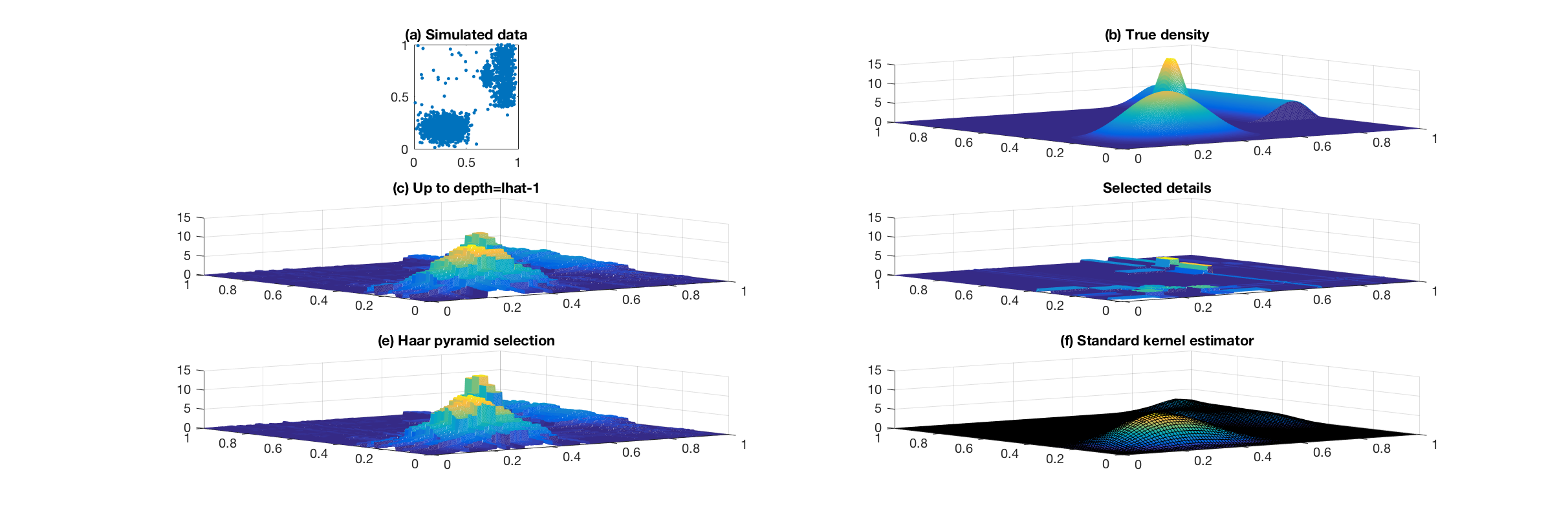}
\caption{Pyramid selection and standard kernel for an example of mixture of multiplicative densities.}
\label{fig:densexample}
\end{figure}
\end{centering} 

In the copula density estimation framework, we consider an example where the copula of $X_i=(X_{i1},X_{i2})$ is either a Frank copula or a Clayton copula conditionally to a  binary variable $Z.$ More precisely, we consider the mixture copula
$$s(x_1,x_2)=0.5 s_F(x_1,x_2) + 0.5 s_C(x_1,x_2)$$
where $s_F$ is the density of a Frank copula with parameter 4 and $s_C$ is the density of a Clayton copula with parameter 2. These two examples of Archimedean copula densities are shown in Figure~\ref{fig:FCcopula} and the resulting mixture in Figure~\ref{fig:copulaexample} (b). We use the same penalty as in the previous example, adapted of course to the copula density estimation framework. We illustrate in Figure~\ref{fig:copulaexample} the pyramid selection procedure on a sample with size $n=2000.$ Though not all theoretical conditions are fully satisfied here, the pyramid selection procedure still provides a reliable estimator.

\begin{figure}[h]
\includegraphics[scale=0.2]{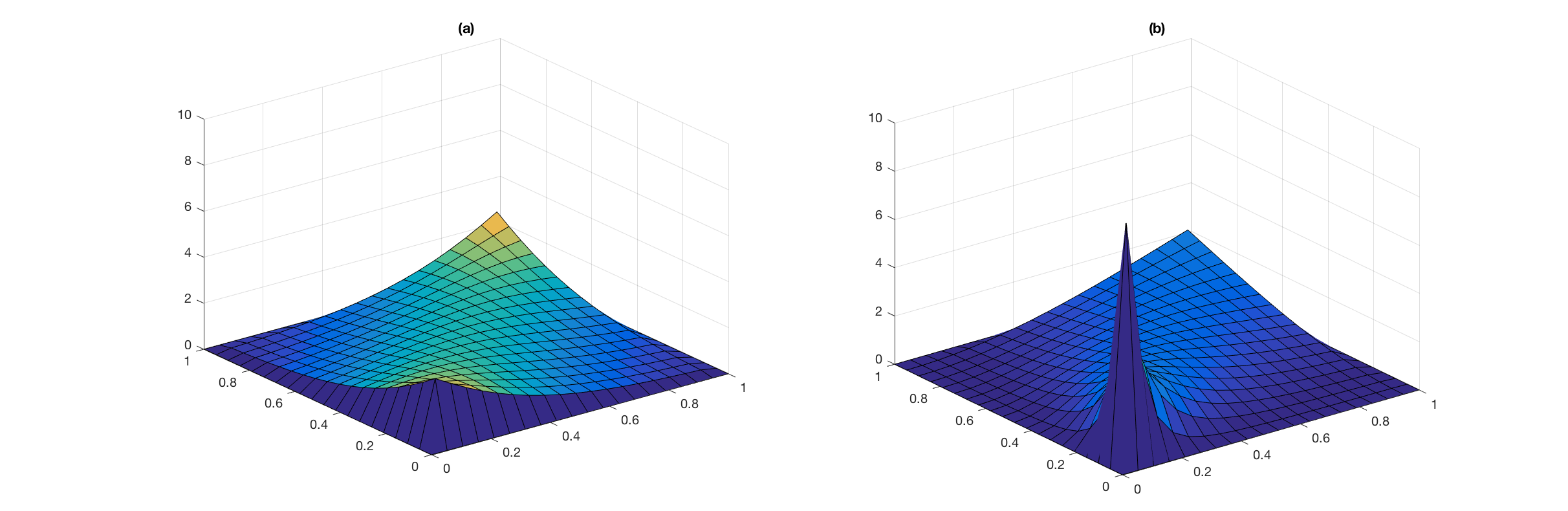}
\caption{Left: Frank copula density with parameter 4; Right: Clayton copula density with parameter 2.}
\label{fig:FCcopula}
\end{figure}

\begin{figure}[h]
\includegraphics[scale=0.2]{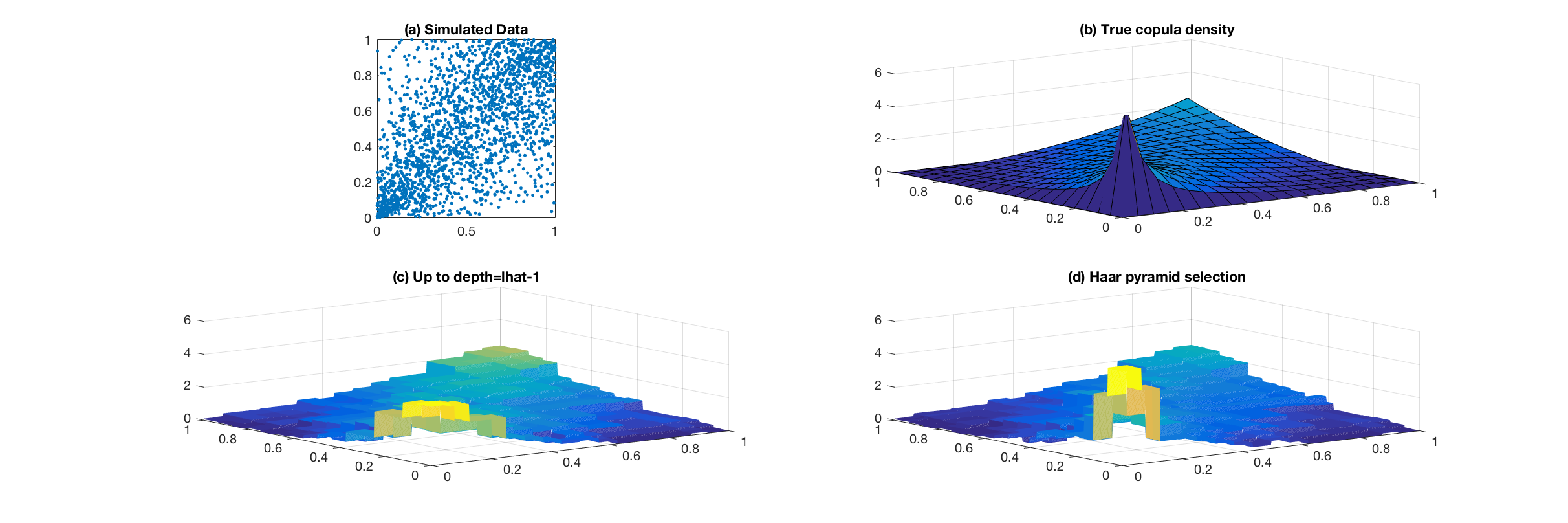}
\caption{Pyramid selection for an example of mixture copula density.}
\label{fig:copulaexample}
\end{figure}

As a conclusion, those examples suggest that the Haar pyramid selection already provides a useful new estimation procedure. This is most encouraging for pyramid selection based on higher order wavelets, whose full calibration based on an extensive simulation study in each framework will be the subject of another work. 

\section{Proofs}\label{sec:proofs}
We shall use repeatedly the classical inequality
\begin{equation}\label{eq:basic}
2ab \leq \theta a^2 + \frac1\theta b^2
\end{equation}
for all positive $\theta,a,b.$

\subsection{Proof of Proposition~\ref{prop:copulaone}}\label{sec:proofscopula}
We only have to prove~\ref{prop:copulaone} for $m_\bullet.$ Indeed, as any pyramidal model $m$ is a subset of $m_\bullet,$ a common upper-bound for the residual terms is 
$$ \|\check s^\star_m-\hat s^\star_m\|^2\leq \|\check s^\star_{m_\bullet}-\hat s^\star_{m_\bullet}\|^2.$$

Under Assumption (L), and thanks to assumptions $S.iii)$ and $W.v),$ we have that for all $\bm\lambda\in m_\bullet,$ $\bm x=(x_1,\ldots,x_d)$ and $\bm y=(y_1,\ldots,y_d) \in Q,$ 
$$|\Psi_{\bm\lambda}(\bm x)-\Psi_{\bm\lambda}(\bm y)| \leq \kappa^d 2^{3L_\bullet/2}\sum_{k=1}^d |x_k-y_k|.$$ 
According to Massart's version of Dworetzky-Kiefer-Wolfowitz inequality (see~\cite{MassartDKW}), for any positive $z,$ and $1\leq k\leq d,$ there exists some event $\Omega_k(z)$ on which $\|\hat F_{nk}-F_k\|_\infty \leq z/\sqrt{n}$ and such that $\P(\Omega^c_k(z))\leq 2\exp(-2z^2).$ Setting $\Omega(z)=\cap_{k=1}^d \Omega_k(z),$ we thus have for all $\bm\lambda\in m_\bullet$
$$|\check \beta_{\bm\lambda}- \widehat\beta_{\bm\lambda}| \leq 
\kappa^d 2^{3L_\bullet/2} d(z/\sqrt{n}) \BBone_{\Omega(z)} + \kappa^d 2^{3L_\bullet/2} d \BBone_{\Omega^c(z)},$$
hence 
$$ \E[\|\check s^\star_{m_\bullet}-\hat s^\star_{m_\bullet}\|^2] \leq \kappa^{2d} 2^{3L_\bullet} d^2(z^2/n + 2d\exp(-2z^2)) D_{m_\bullet}.$$
Finally, $D_{m_\bullet}$ is of order $L^{d-1}_\bullet 2^{L_\bullet}$ (see Proposition~\ref{prop:combinatorial}), so by choosing $2z^2=\log(n),$
$$ \E[\|\check s^\star_{m_\bullet}-\hat s^\star_{m_\bullet}\|^2] \leq C(\kappa,d)L^{d-1}_\bullet 2^{4L_\bullet} \log(n)/n.$$

\subsection{Proof of Proposition~\ref{prop:disclevyone}}\label{sec:proofslevydisc}
For all bounded measurable function $g,$ let us denote  $D_\Delta(g)=\E[g(X_\Delta)]/\Delta-\int_Q g s.$ For all $\bm\lambda\in \Lambda,$ 
\begin{equation}\label{eq:disclevycoeff}
\E\left[\left(\check  \beta_{\bm\lambda}-\widehat\beta_{\bm\lambda}\right)^2\right] 
\leq 4\left( \Var(\check  \beta_{\bm\lambda})+ \Var(\widehat\beta_{\bm\lambda})+ D^2_\Delta(\Psi_{\bm\lambda})\right)
\leq 8\frac{\|s\|_\infty}{n\Delta}+ 4\frac{D_\Delta(\Psi^2_{\bm\lambda})}{n\Delta}+ 4D^2_\Delta(\Psi_{\bm\lambda}).
\end{equation}

We shall bound $D_\Delta(\Psi_{\bm\lambda})$ by using the decomposition of a L\'evy process into a big jump compound Poisson process and an independent small jump L\'evy process. Let us fix $\varepsilon>0$ small enough so that $Q=\prod_{k=1}^d [a_k,b_k]\subset \{\|x\|>\varepsilon\}$ and denote by $(\Sigma,\bm{\mu},\nu)$ the characteristic L\'evy triplet of  $\mathbf X=(X_t)_{t\geq 0},$ where $\bm{\mu}$ stands for the drift and $\nu$ is the L\'evy measure, with density $f$ with respect to the Lebesgue measure on $\R^d$ (see Section~\ref{sec:examples}). Then $\mathbf X$ is distributed as $\mathbf X^\varepsilon +\tilde {\mathbf  X}^\varepsilon,$ where 
$\mathbf X^\varepsilon$ and $\tilde{\mathbf X}^\varepsilon$ are independent L\'evy processes  with following characteristics. First, $\mathbf X^\varepsilon$  is a L\'evy process with characteristic L\'evy triplet $(\Sigma,\bm{\mu}_\varepsilon,\nu_\epsilon),$ where the drift is 
$$\bm{\mu}_\varepsilon= \bm{\mu}-\int_{\varepsilon<\|x\|\leq 1} x\ f(x)\d x.$$
and the L\'evy measure is
$$\nu_\varepsilon (\d x)=\BBone_{\|x\|\leq \varepsilon} f(x) \d x.$$
The process $\tilde {\mathbf  X}^\varepsilon$ is the compound Poisson process
$$\tilde {X}_t^\varepsilon=\sum_{i=1}^{\tilde N_t} \xi_i,$$
where $\tilde N$ is a homogeneous Poisson process with intensity $\lambda_\varepsilon=\nu(\{\|x\|> \varepsilon\}),$  $(\xi_i)_{i\geq 1}$ are i.i.d. with density $\lambda_\varepsilon^{-1} \BBone_{\|x\|> \varepsilon}f(x),$ and $\tilde N$ and $(\xi_i)_{i\geq 1}$ are independent.

Conditioning by $\tilde N$ and using the aforementioned independence properties yields
$$\frac{\E[\Psi_{\bm\lambda}(X_\Delta)]}{\Delta}=e^{-\lambda_{\varepsilon} \Delta} \frac{\E[\Psi_{\bm\lambda}(X^\varepsilon_\Delta)]}{\Delta} +  \lambda_{\varepsilon}e^{-\lambda_{\varepsilon} \Delta} \E[\Psi_{\bm\lambda}(X^\varepsilon_\Delta+\xi_1)]+ \lambda^2_{\varepsilon} \Delta e^{-\lambda_{\varepsilon} \Delta} \sum_{j=0}^\infty\E\left[\Psi_{\bm\lambda}\left(X^\varepsilon_\Delta+\sum_{i=1}^{j+2}\xi_i\right)\right]\frac{(\lambda_{\varepsilon} \Delta)^j}{(j+2)!}.$$
Conditioning by $\xi_1$ and using independence between $X^\varepsilon_\Delta$ and $\xi_1$ then yields 
$$\lambda_{\varepsilon} \E[\Psi_{\bm\lambda}(X^\varepsilon_\Delta+\xi_1)]=\int_{\|x\|>\varepsilon} \E[\Psi_{\bm\lambda}(X^\varepsilon_\Delta+x)]f(x)\d x.$$
Writing $\langle \Psi_{\bm\lambda},s\rangle=  e^{-\lambda_{\varepsilon} \Delta} \langle \Psi_{\bm\lambda},s\rangle+ (1-e^{-\lambda_{\varepsilon} \Delta})\langle \Psi_{\bm\lambda},s\rangle$ and using $(1-e^{-\lambda_{\varepsilon} \Delta})\leq \lambda_{\varepsilon} \Delta$ leads to
$$|D_\Delta(\Psi_{\bm\lambda})| \leq R_\Delta^{(1)}(\Psi_{\bm\lambda}) +R_\Delta^{(2)}(\Psi_{\bm\lambda}) + R_\Delta^{(3)}(\Psi_{\bm\lambda}) +R_\Delta^{(4)}(\Psi_{\bm\lambda}),$$
where
$$ R_\Delta^{(1)}(\Psi_{\bm\lambda}) = e^{-\lambda_{\varepsilon} \Delta}\frac{\E[\Psi_{\bm\lambda}(X^\varepsilon_\Delta)]}{\Delta}, \quad
R_\Delta^{(2)}(\Psi_{\bm\lambda}) =  e^{-\lambda_{\varepsilon} \Delta}\int_{\|x\|>\varepsilon}  \left|\E\left[\Psi_{\bm\lambda}(X^\varepsilon_\Delta+x) -\Psi_{\bm\lambda}(x)\right]\right|f(x)\d x,$$
\begin{equation}\label{eq:R3R4}
R_\Delta^{(3)}(\Psi_{\bm\lambda})=\lambda_{\varepsilon} \Delta \|\Psi_{\bm\lambda}\|_1 \|s\|_\infty,\quad
R_\Delta^{(4)}(\Psi_{\bm\lambda})= \lambda_{\varepsilon}^2 \Delta \| \Psi_{\bm\lambda}\|_\infty.
\end{equation}

As $\Psi_{\bm\lambda}$ has compact support $Q,$
$$ R_\Delta^{(1)}(\Psi_{\bm\lambda}) \leq  e^{-\lambda_{\varepsilon} \Delta} \| \Psi_{\bm\lambda}\|_\infty \frac{\P(X^\varepsilon_\Delta \in Q)}{\Delta}.$$
Let us denote by $X^\varepsilon_{\Delta,k}$ the $k$-th coordinate of $X^\varepsilon_\Delta$ and by $d_Q$ the maximal distance from $[a_k,b_k]$ to 0, for $k=1,\ldots,d,$ reached for instance at $k=k_0.$ 
We deduce from the proof of Lemma 2 in~\cite{RW} (see also~\cite{FigueroaHoudre}, equation (3.3)) that  there exists $z_0=z_0(\varepsilon)$ such that if $\Delta <d_Q/z_0(\varepsilon),$
$$\P(X^\varepsilon_\Delta \in Q)\leq \P(|X^\varepsilon_{\Delta,k_0}| \geq  d_Q)\leq  \exp\left((z_0\log(z_0)+u-u\log(u))/(2\varepsilon)\right) \Delta^{d_Q/(2\varepsilon)}$$
so that 
\begin{equation}\label{eq:R1}
R_\Delta^{(1)}(\Psi_{\bm\lambda}) \leq  C(d_Q,\varepsilon) e^{-\lambda_{\varepsilon} \Delta} \| \Psi_{\bm\lambda}\|_\infty \Delta^{d_Q/(2\varepsilon)-1}.
\end{equation}

Under Assumption $(L),$ $\Psi_{\bm\lambda}$ is Lipschitz on $Q,$ so 
$$\left|\Psi_{\bm\lambda}(X^\varepsilon_\Delta+x) -\Psi_{\bm\lambda}(x)\right|\leq  \|\Psi_{\bm\lambda}\|_L \|X_\Delta^\varepsilon\|_1\BBone_{\{X_\Delta^\varepsilon +x \in Q\} \cap \{x \in Q\}} +|\Psi_{\bm\lambda}(x)| \BBone_{\{X_\Delta^\varepsilon +x \notin Q\} \cap\{x \in Q\}} + \|\Psi_{\bm\lambda}\|_\infty  \BBone_{\{X_\Delta^\varepsilon +x \in Q\} \cap\{x \notin Q\}}.
$$
Besides, as $Q$ is compact and bounded away from the origin, there exists $\delta_Q$ and $\rho_Q>0$ such that 
$$\{X_\Delta^\varepsilon +x \in Q\} \cap \{x \in Q\} \subset \{\|X_{\Delta}^\varepsilon\| \geq \delta_Q\} $$
$$(\{X_\Delta^\varepsilon +x \in Q\} \cap \{x \notin Q\})\cup  (\{X_\Delta^\varepsilon +x \notin Q\} \cap\{x \in Q\}) \subset \{\|X_{\Delta}^\varepsilon\| \geq \rho_Q\}$$ 
The L\'evy measure of $\mathbf X^\varepsilon$ is compactly supported and satisfies 
$$\int \|x\| ^2 \nu_\varepsilon(\d x) = \int_{\|x\| \leq \varepsilon} \|x\| ^2 \nu(\d x) $$
which is finite since $\nu$ is a L\'evy measure (see for instance~\cite{Sato}, Theorem 8.1). So we deduce from~\cite{Millar}, Theorem 2.1, that 
\begin{equation*}\label{eq:momsmalllevy}
\E\left[\|X_\Delta^\varepsilon\|^2\right]\leq C(d,f) \Delta,
\end{equation*}
hence 
$$\E\left[ \|X_\Delta^\varepsilon\|_1\BBone_{\|X_{\Delta}^\varepsilon\| \geq \delta_Q} \right] \leq 
C(d)\delta_Q^{-1} \E\left[ \|X_\Delta^\varepsilon\|^2\right] \leq C(d,f)\delta_Q^{-1} \Delta 
$$
and from Markov inequality
$$ \P(\|X^\varepsilon_\Delta\| \geq \rho_Q) \leq C(d,f) \rho_Q^{-2}\Delta.$$
Finally, fixing $0<\varepsilon<\min(d_Q/4, \inf_{x\in Q} \|x\|),$  we have for all $0<\Delta < \min(d_Q/z_0(\epsilon), 1)$
\begin{equation}\label{eq:R2}
R_\Delta^{(2)}(\Psi_{\bm\lambda}) \leq C(d,f) e^{-\lambda_{\varepsilon}\Delta} \Delta (\delta_Q^{-1} \lambda_\varepsilon \|\Psi_{\bm\lambda}\|_L  + \rho_Q^{-2}\|s\|_\infty \|\Psi_{\bm\lambda}\|_1+ \rho_Q^{-2} \lambda_\varepsilon \|\Psi_{\bm\lambda}\|_\infty ).
\end{equation}

For all $\bm\lambda\in m_\bullet,$ 
$$\max(\|\Psi_{\bm\lambda}\|_L, \|\Psi_{\bm\lambda}\|_\infty, \|\Psi_{\bm\lambda}\|_1) \leq  C(\kappa) 2^{3L_\bullet/2},\max(\|\Psi^2_{\bm\lambda}\|_L, \|\Psi^2_{\bm\lambda}\|_\infty, \|\Psi^2_{\bm\lambda}\|_1) \leq  C(\kappa) 2^{2L_\bullet},$$
so that combining~\eqref{eq:disclevycoeff},~\eqref{eq:R1},~\eqref{eq:R2} and~\eqref{eq:R3R4} yields
$$\E[\|\check s^\star_m-\hat s^\star_m\|^2]
\leq 8\frac{\|s\|_\infty D_m}{n\Delta}+ C(\kappa,d,f,Q,\varepsilon) L_\bullet^{d-1} \frac{2^{4L_\bullet}n\Delta^3+ 2^{3L_\bullet}\Delta }{n\Delta}.$$

%

\subsection{Proof of Proposition~\ref{prop:combinatorial}}\label{sec:proofcombinatorial}
Due to hypotheses $S.ii)$ and $W.ii),$ we have for all $j\geq j_0,$
$$2^{j-1}\leq \sharp \nabla_j \leq M 2^{j-1},$$
hence, for all $\bm j\in \N_{j_0}^d,$
\begin{equation*}\label{eq:cardnablauni}
(1/2)^d 2^{|\bm j|}\leq \sharp \bm\nabla_{\bm j} \leq (M/2)^d 2^{|\bm j|}.
\end{equation*}
Let us fix $\ell\in\{dj_0,\ldots, L_\bullet\}.$ The number of $d$-uples $\bm j\in\N_{j_0}^d$ such that $|\bm j|=\ell$ is equal to the number of partititions of the integer $\ell-dj_0$ into $d$ nonnegative integers, hence
\begin{equation*}\label{eq:cardpartint}
\sharp \bm J_\ell=\binom{\ell-dj_0+d-1}{d-1}=\prod_{k=1}^{d-1} \left(1+\frac{\ell-dj_0}{k}\right).
\end{equation*}
The last two displays and the classical upper-bound for binomial coefficient (see for instance~\cite{Massart}, Proposition 2.5) yield 
\begin{equation}\label{eq:cardbignabla}
c_0(d) (\ell-dj_0+d-1)^{d-1}2^\ell \leq \sharp U\bm\nabla(\ell) \leq c_1(M,d) (\ell-dj_0+d-1)^{d-1}2^\ell,
\end{equation}
where $c_0(d)=2^{-d}(d-1)^{-(d-1)}$ and $c_1(M,d)=(M/2)^d(e/(d-1))^{d-1}.$

Let us now fix $\ell_1\in\{dj_0+1,\ldots, L_\bullet +1\}.$  Any model $m\in\mathcal M^\mathcal P_{\ell_1}$ satisfies
$$D_m
= \sum_{\ell=dj_0}^{\ell_1-1} \sharp U\bm\nabla(\ell) + \sum_{k=0}^{L_\bullet-\ell_1}N(\ell_1,k).$$
So we obviously have 
$$D_m  \geq \sharp U\bm\nabla(\ell_1-1) \geq \kappa_1(d)(\ell_1-dj_0+d-2)^{d-1}2^{\ell_1},$$
with $\kappa_1(d)=c_0(d)/2=2^{-(d+1)}(d-1)^{-(d-1)}.$ 
Besides, with our choice of $N(\ell_1,k),$ 
\begin{equation*}
D_m \leq c_1(M,d) (\ell_1-dj_0+d-2)^{d-1} \sum_{\ell=dj_0}^{\ell_1-1}2^\ell + 2M^{-d}c_1(M,d) s_1(d) (\ell_1-dj_0+d-2)^{d-1}2^{\ell_1},
\end{equation*}
so that Proposition~\ref{prop:combinatorial} holds with  $\kappa_2(d,j_0,B) = c_1(M,d) (1+2M^{-d}c_1(M,d) s_1(d)),$ where
$$s_1(d)=\sum_{k=0}^\infty \frac{(1+k/(d-1))^{d-1}} {(2+k)^{d+2}}.$$

The number of subsets of $\Lambda$ in $\mathcal M^\mathcal P_{\ell_1}$ satisfies
\begin{equation*}
\sharp\mathcal M^\mathcal P_{\ell_1}
= \prod_{k=0}^{L_\bullet-\ell_1} \binom{\sharp U\bm{\nabla}(\ell_1+k)}{N(\ell_1,k)}
\leq \prod_{k=0}^{L_\bullet-\ell_1} \left(\frac{e\: \sharp U\bm{\nabla}(\ell_1+k)}{N(\ell_1,k)}\right)^{N(\ell_1,k)}.
\end{equation*}
For $k\in\{0,\ldots,L_\bullet-\ell_1\},$ let $f(k)=(k+2)^{d+2}2^k M^d/2,$  then $N(\ell_1,k)\leq \sharp U\bm{\nabla}(\ell_1+k)/f(k).$ As the function $x\in [0,U]\mapsto x \log(e U/x)$ is increasing, we deduce 
\begin{equation*}
\log (\sharp \mathcal M^\mathcal P_{\ell_1}) 
\leq  D(\ell_1)  \sum_{k=0}^{L_\bullet-\ell_1} \frac{\sharp U\bm\nabla(\ell_1+k)}{\sharp U\bm\nabla(\ell_1-1)} \frac{1+\log(f(k))}{f(k)}.
\end{equation*}
Setting 
$$s_2=\sum_{k=0}^\infty\frac{1}{(k+2)^3},\quad s_3=\sum_{k=0}^\infty\frac{\log(k+2)}{(k+2)^3}, s_4 = \sum_{k=0}^\infty\frac{1}{(k+2)^2},$$ one may take for instance
$\kappa_3(j_0,B,d)=(\log(e/2)+d\log(M))s_2 + (d+2) s_3+\log(2) s_4$ in Proposition~\ref{prop:combinatorial}.

\subsection{Proof of Theorem~\ref{theo:choosepen}}\label{sec:proofpen}
\subsubsection{Notation and preliminary results}

Hyperbolic wavelet bases inherit from the underlying univariate wavelet bases a localization property which can be stated as follows.
\begin{lemm}\label{lemm:localization} Let $\underline D(L_\bullet)=\left(e(L_\bullet-dj_0+d-1)/(d-1)\right)^{d-1} 2^{L_\bullet/2},$ then for all real-valued sequence $(a_{\bm\lambda})_{\bm\lambda\in m_\bullet},$
$$\max\left\{\left\|\sum_{\bm\lambda\in m_\bullet} a_{\bm\lambda} \Psi_{\bm\lambda}\right\|_\infty,\left\|\sum_{\bm\lambda\in m_\bullet} a_{\bm\lambda} \Psi^\star_{\bm\lambda}\right\|_\infty\right\} \leq \kappa'_7 \max_{{\bm\lambda}\in m_\bullet} |a_{\bm\lambda}|  \underline D(L_\bullet),$$
where $\kappa'_7=\kappa^{2d} (2+\sqrt{2})$ for instance.
\end{lemm}
\begin{proof} For all $\mathbf x=(x_1,\ldots,x_d)\in[0,1]^d,$ using assumptions $S.vi), S.vii), S.viii),W.iv), W.v), W. vi)$ in Section~\ref{sec:uniwave_assumptions}, we get
\begin{align*}
\left|\sum_{\bm\lambda\in m_\bullet} a_{\bm\lambda} \Psi_{\bm\lambda}\right|
&\leq \max_{\bm\lambda\in m_\bullet} \left| a_{\bm\lambda} \right| \sum_{\ell=dj_0}^{L_\bullet} \sum_{\bm j \in \bm J_\ell} \prod_{k=1}^d \left(\sum_{\lambda_k\in\nabla_{j_k}}|\psi_{\lambda_k}(x_k)|\right)\\
&\leq \kappa^{2d} \max_{\bm\lambda\in m_\bullet} \left| a_{\bm\lambda} \right| \sum_{\ell=dj_0}^{L_\bullet} \sum_{\bm j \in \bm J_\ell} 2^{\ell/2}.
\end{align*}
We deduce from the proof of Proposition~\ref{prop:combinatorial} the upper-bound
$\sharp\bm J_\ell\leq \left(e(L_\bullet-dj_0+d-1)/(d-1)\right)^{d-1}$ which allows to conclude.
\end{proof}

For all $t\in\L_2([0,1]^d),$  we define
$$\nu(t)=\sum_{\bm\lambda\in\Lambda}\langle t,\Psi_{\bm\lambda} \rangle (\check \beta_{\bm\lambda}-\langle s,\Psi_{\bm\lambda} \rangle),\quad
\nu_R(t)=\sum_{\bm\lambda\in\Lambda}\langle t,\Psi_{\bm\lambda} \rangle (\hat \beta_{\bm\lambda}-\check \beta_{\bm\lambda}),\quad
\hat\nu(t)=\nu(t)+\nu_R(t),$$
and for all $m\in\mathcal M^{\mathcal P},$ we set 
$$\chi(m)=\sup_{t\in S^\star_m|\|t\|_{\bm\Psi}=1} \nu(t), \quad \chi_R(m)=\sup_{t\in S^\star_m|\|t\|_{\bm\Psi}=1} \nu_R(t).$$

\begin{lemm}\label{lemm:chi2exp}
For all $m\in\mathcal M^{\mathcal P},$ let $t^\star_m= \sum_{\bm\lambda\in m}(\nu(\Psi^\star_{\bm\lambda})/\chi(m)) \Psi^\star_{\bm\lambda},$ then 
$$\chi(m)=\sqrt{\sum_{\bm\lambda\in m}\nu^2(\Psi^\star_{\bm\lambda})}=\|s^\star_m-\check s^\star_m\|_{\bm\Psi}=\nu(t^\star_m),$$
$$\chi_R(m)=\sqrt{\sum_{\bm\lambda\in m}\nu_R^2(\Psi^\star_{\bm\lambda})}=\|\check s^\star_m-\hat s^\star_m\|_{\bm\Psi}.$$
\end{lemm}

\begin{proof} The proof follows from the linearity of $\nu$ and $\nu_R$ and Cauchy-Schwarz inequality.
\end{proof}

\begin{lemm}\label{lemm:chitrunc}Let $\epsilon=\kappa'_2\|s\|_\infty/(\kappa'_3\kappa'_7\underline D(L_\bullet))$ and 
$$\Omega_T=\cap_{\bm \lambda\in m_\bullet}\left\{|\nu(\Psi^\star_{\bm\lambda})| \leq \epsilon \right\}.$$
For all $x>0,$ there exists a measurable event $\Omega_m(x)$ on which
$$\chi^2(m) \BBone_{\Omega_T\cap \Omega_\sigma} \leq 2 \kappa'^2_1\kappa'_5 \sum_{\bm\lambda \in m} \frac{\max\{\hat\sigma^2_{\bm\lambda},1\}}{\bar n} + 8 \kappa'_2 \|s\|_\infty\frac{x}{\bar n}.$$
and such that $\P(\Omega^c_m(x))\leq \exp(-x).$
\end{lemm}

\begin{proof} We observe that $\chi(m)=\mathcal Z (\mathcal T_m)$ where $\mathcal T_m=\{t\in S^\star_m | \|t\|_{\bm\Psi} =1\}.$ Let us set $z=\sqrt{\kappa'_2\|s\|_\infty x/\bar n}$ and consider a countable and dense subset $\mathcal T'_m$ of $\{t\in S^\star_m | \|t\|_{\bm\Psi} =1, \max_{\bm\lambda\in m} |\langle t,\Psi_{\bm\lambda}\rangle|\leq \epsilon/z\}.$ Thanks to the localization property in Lemma~\ref{lemm:localization},
$$\sup_{t\in\mathcal T'_m}\left\|\sum_{\lambda \in m_\bullet} \langle t, \Psi_{\bm\lambda} \rangle \Psi_{\bm \lambda}\right \|_\infty  \leq \kappa'_3\frac{\sqrt{\kappa'_2 \|s\|_\infty}}{\sqrt{x/\bar n}}.$$
So Assumption \textbf{(Conc)} ensures that there exists $\Omega_m(x)$  such that $\P(\Omega^c_m(x))\leq \exp(-x)$ and on which
$$\mathcal Z(\mathcal T'_m) \leq \kappa'_1 \E\left[\mathcal Z(\mathcal T'_m)\right] + 2\sqrt{\kappa'_2 \|s\|_\infty \frac{x}{\bar n}},$$
hence
$$\mathcal Z^2(\mathcal T'_m) \leq 2\kappa'^2_1 \E^2\left[\mathcal Z(\mathcal T'_m)\right] + 8\kappa'_2 \|s\|_\infty \frac{x}{\bar n}.$$
As $Z(\mathcal T'_m)\leq \chi(m),$ we obtain by convexity and Lemma~\ref{lemm:chi2exp}
$$\E^2\left[\mathcal Z(\mathcal T'_m)\right]\leq \E\left[\chi^2(m)\right] =\sum_{\bm\lambda \in m} \Var(\check \beta_{\bm\lambda}).$$
On $\Omega_T\cap \{\chi(m) \geq z\},$ $t^\star_m$ given by Lemma~\ref{lemm:chi2exp} satisfies $\sup_{\bm\lambda\in m}|\langle t^\star_m, \Psi_{\bm\lambda}\rangle| \leq \epsilon/z,$ so that $\chi^2(m)=\mathcal Z^2(\mathcal T'_m),$ while on $\Omega_T\cap \{\chi(m) < z\},$ $\chi^2(m) < \kappa'_2\|s\|_\infty x/\bar n.$ The proof then follows from Assumption \textbf{(Var)}.
\end{proof}

\subsubsection{Proof of Theorem~\ref{theo:choosepen}}

Let us fix $m\in\mathcal M^{\mathcal P}.$ From the definition of $\hat m^{\mathcal P}$ and of $\hat s^\star_m,$ we get 
$$\hat\gamma(\tilde s^{\mathcal P})+\pen(\hat m^{\mathcal P})\leq  \hat\gamma(s^\star_m)+\pen(m).$$
For all $t,u\in\L_2([0,1]^d),$ 
$$\hat\gamma(t)-\hat\gamma(u)=\|t-s\|^2_{\bm \Psi}-\|u-s\|^2_{\bm \Psi}-2\hat \nu(t-u),$$
so
$$\|s-\tilde s^{\mathcal P}\|^2_{\bm \Psi}\leq \|s-s^\star_m\|^2_{\bm \Psi}+2\hat \nu(\tilde s^{\mathcal P}-s^\star_m)+ \pen(m)-\pen(\hat m^{\mathcal P}).$$
Using the triangle inequality and Inequality~\eqref{eq:basic} with $\theta=1/4$ and $\theta=1,$ we get
\begin{align*}
2\hat \nu(\tilde s^{\mathcal P}-s^\star_m)
&\leq 2\|\tilde s^{\mathcal P}-s^\star_m\|_{\bm \Psi}(\chi(m\cup\hat m^{\mathcal P})+\chi_R(m\cup\hat m^{\mathcal P}))\\
&\leq \frac{1}{2} \|s-\tilde s^{\mathcal P}\|^2_{\bm \Psi} + \frac{1}{2} \|s-s^\star_m\|^2_{\bm \Psi} + 8 \chi^2(m\cup\hat m^{\mathcal P}) + 8\chi^2_R(m\cup\hat m^{\mathcal P}),
\end{align*}
hence 
\begin{equation}\label{eq:step1}
\|s-\tilde s^{\mathcal P}\|^2_{\bm \Psi}\leq 3\|s-s^\star_m\|^2_{\bm \Psi}+16\chi^2(m\cup\hat m^{\mathcal P})+2( \pen(m)-\pen(\hat m^{\mathcal P})) + 16\chi^2_R(m\cup\hat m^{\mathcal P}).
\end{equation}

Let us fix $\zeta>0$ and set $\omega=\kappa_3(j_0,B,d)+\log(2)$ and $\Omega_\star(\zeta)=\cap_{m'\in\mathcal M^{\mathcal P}} \Omega_{m\cup m'}(\zeta + \omega D_m).$ We deduce from Lemma~\ref{lemm:chitrunc} that on $\Omega_\star(\zeta)$
\begin{equation}\label{eq:step2}
\chi^2(m\cup \hat m^{\mathcal P}) \BBone_{\Omega_T\cap \Omega_\sigma} \leq 2 \kappa'^2_1\kappa'_5  \sum_{\bm\lambda \in m\cup \hat m^{\mathcal P}}\frac{\max\{\hat\sigma^2_{\bm\lambda},1\}}{\bar n}+ 8 \kappa'_2 \|s\|_\infty\frac{\omega (D_m+D_{\hat m^{\mathcal P}})}{\bar n}+ 8 \kappa'_2 \|s\|_\infty\frac{\zeta}{\bar n}.
\end{equation}
Besides, given Proposition~\ref{prop:combinatorial}, our choice of $\omega$ leads to
$$ \P(\Omega^c_\star(\zeta)) \leq e^{-\zeta} \sum_{\ell=dj_0+1}^{L_\bullet+1} \exp\left(-D(\ell)\left(\omega-\frac{\log(\sharp \mathcal M_\ell^{\mathcal P})}{D(\ell)}\right) \right)\leq e^{-\zeta}.$$
Choosing for instance 
$$\pen(m)=c_1\sum_{\bm\lambda \in m} \frac{\hat\sigma^2_{\bm\lambda}}{\bar n} + c_2\frac{\bar R D_m}{\bar n},$$
with $c_1\geq 16 \kappa'^2_1\kappa'_5$ and $c_2\geq 64 \kappa'_2\omega+8\kappa'_6$ and integrating with respect to $\zeta>0,$ we deduce from~\eqref{eq:step1},~\eqref{eq:step2}, Assumption \textbf{(Var)} and Assumption \textbf{(Conc)} that
\begin{equation}\label{eq:step3}
\E\left[\|s-\tilde s^{\mathcal P}\|^2_{\bm \Psi} \BBone_{\Omega_T\cap \Omega_\sigma} \right] \leq 3 \|s-s^\star_m\|^2_{\bm \Psi}+ C \frac{\bar R D_m}{\bar n}
+ 64\kappa'_2\frac{ \|s\|_\infty}{\bar n}+8\frac{w(\bar n)}{\bar n},
\end{equation}
where $C$ may depend on $\kappa'_1,\kappa'_2,\kappa'_4,\kappa'_5,\kappa'_6,c_1,c_2.$

In order to bound $\E\left[\|s-\tilde s^{\mathcal P}\|^2_{\bm \Psi} \BBone_{\Omega^c_T\cup \Omega^c_\sigma} \right],$ we first notice that from the triangle inequality and Lemma~\ref{lemm:chi2exp}
\begin{align*}
\|s-\tilde s^{\mathcal P}\|_{\bm \Psi}
&\leq \|s- s^\star_{\hat m^\mathcal P}\|_{\bm \Psi}+\|s^\star_{\hat m^\mathcal P}-\hat s^\star_{\hat m^\mathcal P}\|_{\bm \Psi}\\
&\leq   \|s\|_{\bm \Psi}+\chi(\hat m) +\chi_R(\hat m),
\end{align*}
hence 
$$\|s-\tilde s^{\mathcal P}\|^2_{\bm \Psi} \leq   \|s\|^2_{\bm \Psi}+4\chi^2(m_\bullet) +4\chi_R^2(m_\bullet).$$
Then setting $p_T=\P(\Omega_T^c)$ and $p_\sigma=\P(\Omega_\sigma^c),$ Cauchy-Schwarz inequality entails 
\begin{equation*}
\E\left[\|s-\tilde s^{\mathcal P}\|^2_{\bm \Psi} \BBone_{\Omega^c_T\cup \Omega^c_\sigma}\right] \leq  2(p_T+p_\sigma)\|s\|^2_{\bm \Psi}+4\sqrt{p_T+p_\sigma} \left(\sqrt{\E\left[\chi^4(m_\bullet)\right]} +\sqrt{\E\left[\chi_R^4(m_\bullet)\right]}\right).
\end{equation*}
Let $\bm\lambda\in m_\bullet,$ $\|\Psi^\star_{\bm\lambda}\|_\infty \leq \kappa^d 2^{L_\bullet/2},$ so applying Assumption \textbf{(Conc)} with $\mathcal T= \{\Psi^\star_{\bm\lambda}\}$ and $\mathcal T= \{-\Psi^\star_{\bm\lambda}\},$ we get
$$\P\left(|\nu(\Psi^\star_{\bm\lambda})|\geq \epsilon \right)
 \leq 2 \exp\left(-\min\left\{\frac{\bar n \epsilon^2}{4\kappa'_2\|s\|_\infty}, \frac{\bar n \epsilon}{2\kappa'_3 \kappa^d 2^{L_\bullet/2}}\right\}\right).$$
Then setting  $\iota= \left(e(L_\bullet-dj_0+d-1)/(d-1)\right)^{d-1},$ Proposition~\ref{prop:combinatorial} yields
$$p_T\leq 2 \iota 2^{L_\bullet}  \exp\left(-C\|s\|_\infty \frac{\bar n}{\iota^2 2^{L_\bullet}}\right) \leq \frac{C}{\bar n^2 (\log(\bar n)/d)^{d+1}},$$
where $C$ may depend on $\kappa'_2,\kappa'_3,\kappa'_7,j_0,d.$
Besides, we deduce from Assumption \textbf{(Conc)} and Lemma~\ref{lemm:localization} that, for all $x>0,$
$$\P\left(\chi(m_\bullet) \geq \kappa'_1 \sqrt{\frac{\|s\|_\infty D_{m_\bullet}}{\bar n}} + \sqrt{\kappa'_2 \|s\|_\infty \frac{x}{\bar n}} + \kappa'_3 \kappa'_7\underline D(L_\bullet) \frac{x}{\bar n}\right) \leq \exp(-x).$$
For a nonnegative random variable $U,$ Fubini's inequality implies 
$$\E[U^4]=\int_0^\infty 4 x^{p-1} \P(U\geq x)\d x$$
so 
$$\E[\chi^4(m_\bullet)] \leq C \max\left\{\frac{\iota^4 2^{2L_\bullet}}{\bar n^4},\frac{\iota^2 2^{2L_\bullet}}{\bar n^2}\right\}\leq \frac{C}{\left(\log (\bar n)/d\right)^{2(d+1)}}$$
where $C$ may depend on $\kappa'_1,\kappa'_2,\kappa'_3,\kappa'_7,j_0,d.$
Remembering~\eqref{eq:step3} , we conclude that
$$\E\left[\|s-\tilde s^{\mathcal P}\|^2_{\bm \Psi} \right] \leq 3 \|s-s^\star_m\|^2_{\bm \Psi}+ C_1 \frac{\bar R D_m}{\bar n} +C_2 \frac{ \|s\|_\infty}{\bar n}+C_3\max\{\|s\|^2_{\bm \Psi},1\} \left(\frac{1}{\bar n\left(\log (\bar n)/d\right)^{3(d+1)/2}}+\frac{w(\bar n)}{\bar n}\right),$$
where $C_1$ may depend on $\kappa'_1,\kappa'_2,\kappa'_4,\kappa'_5,\kappa'_6,c_1,c_2,$ $C_2$ may depend on $\kappa'_2,$  $C_3$ may depend $\kappa'_1,\kappa'_2,\kappa'_3,\kappa'_7,j_0,d.$

\subsection{Proofs of Corollaries~\ref{corol:density} to~\ref{corol:levydensitydisc}}\label{sec:proofcorollaries}

\subsubsection{Proof of Corollary~\ref{corol:density}}
Assumption \textbf{(Conc)} is a straightforward consequence of Talagrand's inequality, as stated for instance in~\cite{Massart} (Inequality $(5.50),$ and is satisfied, whatever $\theta>0,$ for
\begin{equation}\label{ref:concdensity}
\bar n=n, \kappa'_1=1+\theta,\kappa'_2=2,\kappa'_3=(1/3+1/\theta)/2.
\end{equation}
For all $\bm\lambda\in m_\bullet,$ $\hat \sigma^2_{\bm\lambda}$ is an unbiased estimator for $\Var(\Psi_{\bm\lambda}(Y_1)).$ Besides, the existence of $\Omega_\sigma$ follows from Lemma 1 in~\cite{ReynaudRT} with $\gamma=2.$ Thus Assumptions \textbf{(Var)} and \textbf{(Rem)} are satisfied by taking $\kappa'_4=1,\kappa'_5$ that only depends on $\kappa$ and $d,$ $\kappa'_6=0,$ $w(n)=C(\kappa,j_0,d)/\log^{d+1} (n).$

\subsubsection{Proof of Corollary~\ref{corol:copula}} Setting $Y_i=(F_1(X_{i1}),\ldots,F_d(X_{id})),i=1,\ldots,n,$ we recover the previous density estimation framework, so Assumption \textbf{(Conc)} is still satisfied with~\eqref{ref:concdensity}. Setting $\hat Y_i=(F_{n1}(X_{i1}),\ldots,\hat F_{nd}(X_{id})),i=1,\ldots,n,$ and 
$$\check \sigma^2_{\bm \lambda}=\frac{1}{n(n-1)}\sum_{i=2}^n\sum_{j=1}^{i-1}\left(\Psi_{\bm\lambda}(Y_i)-\Psi_{\bm\lambda}(Y_j)\right)^2,$$
we observe that,  for all $\bm\lambda \in m_{\bullet}$
$$\max\left\{\hat\sigma^2_{\bm\lambda} - 4 \check \sigma^2_{\bm \lambda}, \check \sigma^2_{\bm\lambda} - 4 \hat \sigma^2_{\bm \lambda}\right\}\leq 8 R_{\bm\lambda}(n)$$
where 
$$R_{\bm \lambda}(n)=\frac{1}{n(n-1)}\sum_{i=2}^n(i-1)\left(\Psi_{\bm\lambda}(\hat Y_i)-\Psi_{\bm\lambda}(Y_i)\right)^2.$$
Using the same arguments as in the proof of Proposition~\ref{prop:copulaone}, we get for all $\bm\lambda \in m_{\bullet}$ and all $m\subset m_\bullet$
$$\E\left[R_{\bm \lambda}(n)\right] \leq C(\kappa,d) 2^{3L_\bullet} \log(n) /n,$$
$$R_{\bm \lambda}(n)\leq C(\kappa,d) 2^{3L_\bullet} \log(n) /n$$
except on a set with probability smaller than $2d/n,$
$$\E\left[\|\check s^\star_m - \hat s^\star _m\|^2\right] \leq C(\kappa,d,j_0) L_\bullet^{d-1}2^{4 L_\bullet} \log(n)/n,$$
and 
$$\sqrt{\E\left[\|\check s^\star_{m_\bullet} - \hat s^\star _{m_\bullet}\|^4\right]} \leq C(\kappa,d,j_0) L_\bullet^{d-1}2^{4 L_\bullet} /\sqrt{n}.$$
Building on the proof of Corollary~\ref{corol:density}, we conclude that  Assumptions \textbf{(Var)} and \textbf{(Rem)} are satisfied with $\kappa'_4,\kappa'_5$ that only depend on $\kappa,j_0,d,$ $\kappa'_6=0,$ and $w(n)=\sqrt{n} \log^{d-1}(n).$

\subsubsection{Proof of Corollary~\ref{corol:Poisson}}
Assumption \textbf{(Conc)} is a straightforward consequence of Talagrand's inequality for Poisson processes proved by~\cite{ReynaudPoisson} (Corollary 2), and is satisfied, whatever $\theta>0,$ by
\begin{equation}\label{ref:concdensity}
\bar n=\text{Vol}_d(Q), \kappa'_1=1+\theta,\kappa'_2=12,\kappa'_3=(1.25+32/\theta).
\end{equation}
For all $\bm\lambda\in m_\bullet,$ $\hat \sigma^2_{\bm\lambda}$ is an unbiased estimator for $\int_Q \Psi^2_{\bm\lambda}s= \text{Vol}_d(Q)\Var(\check \beta_{\bm\lambda}).$ Besides, the existence of $\Omega_\sigma$ follows from Lemma 6.1 in~\cite{ReynaudR}. Thus Assumptions \textbf{(Var)} and \textbf{(Rem)} are satisfied by taking $\kappa'_4=1,\kappa'_5$ that only depends on $\kappa$ and $d,$ $\kappa'_6=0,$ $w(\bar n)=C(\kappa,j_0,d)/\log^{d+1} (\bar n).$

\subsubsection{Proof of Corollary~\ref{corol:levydensitycont}}
The proof is similar to that of Corollary~\ref{corol:Poisson} with $\bar n=T.$

\subsubsection{Proof of Corollary~\ref{corol:levydensitydisc}} Regarding Assumption \textbf{(Conc)}, the proof is similar to that of Corollary~\ref{corol:Poisson} with $\bar n=n\Delta.$  For all $\bm\lambda \in m_{\bullet},$ let 
$$\check \sigma^2_{\bm \lambda}=\frac{1}{n\Delta} \iint\limits_{[0,n\Delta]\times Q} \Psi^2_{\bm\lambda}(x) N(\d t, \d x).$$
For any bounded measurable function $g$ on $Q,$ let 
$$R(g)= \int_Q g(\d \widehat M-\d M),\quad I(g)=\int_Q g \d M - \E\left[\int_Q g \d M\right],\quad \hat I(g)=\int_Q g \d \widehat M - \E\left[\int_Q g \d \widehat M\right], $$
then 
$$R(g)=\hat I(g) - I(g) +D_\Delta(g)$$
where $D_\Delta$ has been defined in the proof of Proposition~\ref{prop:disclevyone}.
Notice that
$$\hat \sigma^2_{\bm \lambda} - \check \sigma^2_{\bm \lambda}  = R( \Psi^2_{\bm\lambda})$$
and 
$$\|\hat s^\star_{m_\bullet} - \check s^\star_{m_\bullet}\|^2_{\bm\Psi} = \sum_{\bm\lambda\in m_\bullet} R^2(\Psi_{\bm\lambda}).$$
In the course of the proof of Proposition~\ref{prop:disclevyone}, we have shown that, for bounded and Lipschitz functions $g$ on $Q,$
$$\left|D_\Delta(g)\right|\leq C(\lambda_{\varepsilon},\varepsilon,f,Q) \max\left\{\|g\|_1,\|g\|_\infty,\|g\|_L \right\} \Delta$$
provided $\Delta$ and $\varepsilon$ are small enough. Besides, both $\hat I(g)$ and $I(g)$ satisfy Bernstein inequalities (Bernstein inequality as stated in~\cite{Massart}, Proposition 2.9, for the former, and Bernstein inequality as stated in~\cite{ReynaudPoisson}, Proposition 7, for the latter). Combining all these arguments yields Corollary~\ref{corol:levydensitydisc}.

\subsection{Proof of Proposition~\ref{prop:composite}}\label{sec:proofcomposite}
For $\alpha>0,$ we set $r=\lfloor \alpha \rfloor+1.$

$(i).$ From~\eqref{eq:binomdiff}, it is easy to see that $\Delta^r_{h_\ell,\ell}(f,\mathbf x) = \Delta^r_{h_\ell}(u_\ell,x_\ell).$ Thus $w^{\{\ell\}}_r(f,t_\ell)_p=w_r(u_\ell,t_\ell)_p$  and $w^{\mathbf e}_r(f,\mathbf{t_e})_p=0$ as soon as $\mathbf e \subset \{1,\ldots,d\}$ contains at least two elements. Therefore,
$$\|f\|_{SB^{\alpha}_{p,q,(d)}} \leq C(p) \sum_{\ell=1}^d \|u_\ell\|_{B^{\alpha}_{p,q,(1)}}.$$

$(ii).$ For the sake of readability, we shall detail only two special cases. Let us first deal with the case $f(\mathbf x)=\prod_{\ell=1}^d u_\ell(x_\ell)$  where each $u_\ell\in B^{\alpha}_{p,q,(1)}.$ From~\eqref{eq:binomdiff},
$$\Delta^{r,\mathbf e}_{\mathbf h}(f,\mathbf x)=\prod_{\ell \in \mathbf e} \Delta^r_{h_\ell}(u_\ell,x_\ell) \prod_{\ell \notin \mathbf e} u_\ell(x_\ell),$$ 
so $$\|f\|_{SB^{\alpha}_{p,q,(d)}} \leq 2^d \prod_{\ell=1}^d \|u_\ell\|_{B^{\alpha}_{p,q,(1)}}.$$
Let us now assume that $d=3$ and that $f(\mathbf x)=u_1(x_1)u_{2,3}(x_2,x_3)$ where $u_1\in  B^{\alpha_1}_{p,q,(1)}$ and $u_{2,3}\in  B^{\alpha_2}_{p,q,(2)}.$ We set $r_\ell=\lfloor \alpha_\ell \rfloor+1$ for $\ell=1,2,$ and $\bar r = \lfloor \bar\alpha \rfloor+1,$ where $\bar \alpha=\min(\alpha_1,\alpha_2/2).$ For $0<t_1,t_2,t_3<1,$ we easily have
$$\|f\|_p=\|u_1\|_p\|u_{2,3}\|_p$$
$$t_1^{-\bar \alpha} w_{\bar r}^{\{1\}}(f,t_1)_p \leq t_1^{-\alpha_1} w_{r_1}(u_1,t_1)_p\|u_{2,3}\|_p$$
$$t_\ell^{-\bar\alpha} w_{\bar r}^{\{\ell\}}(f,t_\ell)_p \leq \|u_1\|_p t_\ell^{-\alpha_\ell} w_{r_\ell}^{\{\ell\}}(u_{2,3},t_\ell)_p, \text{ for } \ell=2,3$$
$$t_1^{-\bar\alpha}t_\ell^{-\bar\alpha} w_{\bar r}^{\{1,\ell\}}(f,t_1,t_\ell)_p \leq t_1^{-\alpha_1}w_{r_1}(u_1,t_1)_p t_\ell^{-\alpha_\ell} w_{r_\ell}^{\{\ell\}}(u_{2,3},t_\ell)_p, \text{ for } \ell=2,3.$$
Besides, we deduce from~\eqref{eq:binomdiff} that 
$$\|\Delta^{\bar r}_h(g,.)\|_p \leq C({\bar r},p) \|g\|_p,$$
and as operators $\Delta^{\bar r} _{h_\ell,\ell}$ commute, we have
$$t_2^{-\bar\alpha}t_3^{-\bar\alpha} w_{\bar r}^{\{2,3\}}(f,t_2,t_3)_p \leq C(p,\bar r) \|u_1\|_p t_2^{-\bar\alpha}t_3^{-\bar\alpha}  \min\left\{w_{\bar r}^{\{2\}}(u_{2,3},t_2)_p, w_{\bar r}^{\{3\}}(u_{2,3},t_3)_p\right\}.$$
The inequality of arithmetic and geometric means entails that $2t_2^{-\bar\alpha}t_3^{-\bar\alpha}\leq t_2^{-2\bar\alpha}+ t_3^{-2\bar\alpha},$ so 
 $$t_2^{-\bar\alpha}t_3^{-\bar\alpha} w_{\bar r}^{\{2,3\}}(f,t_2,t_3)_p \leq C(p,\bar r) \|u_1\|_p \left(t_2^{-\alpha_2} w_{r_2}^{\{2\}}(u_{2,3},t_2)_p +t_3^{-\alpha_3} w_{r_2}^{\{3\}}(u_{2,3},t_3)_p\right).$$
In the same way,
$$t_1^{-\bar\alpha}t_2^{-\bar\alpha}t_3^{-\bar\alpha} w_{\bar r}^{\{1,2,3\}}(f,t_1,t_2,t_3)_p \leq C(p,\bar r) t_1^{-\alpha_1} w_{r_1}(u_1,t_1)_p \left(t_2^{-\alpha_2} w_{r_2}^{\{2\}}(u_{2,3},t_2)_p +t_3^{-\alpha_3} w_{r_2}^{\{3\}}(u_{2,3},t_3)_p\right).$$
Consequently,
$$\|f\|_{SB^{\bar\alpha}_{p,q,(d)}} \leq C(p,\bar r) \|u_1\|_{B^{\alpha_1}_{p,q,(1)}}  \|u_{2,3}\|_{B^{\alpha_2}_{p,q,(2)}}.$$

$(iii).$ The proof follows from the chain rule for higher order derivatives of a composite function. Notice that for all $1\leq \ell\leq d$ and $1\leq r\leq \alpha-1,$ $u_\ell^{(r)}\in W^{\alpha-r}_{p,(1)},$ with $\alpha-r>1/p,$ so $u_\ell^{(r)}$ is bounded.
 
$(iv).$ The proof follows from a $d$-variate extension of Theorem 4.1, Inequality (10) in~\cite{Potapov} (see also~\cite{DeVoreLorentz} Chapter 6, Theorem 3.1).

$(v).$ See Theorem 3.10 in~\cite{NguyenSickel}.

\subsection{Proof of Theorem~\ref{theo:approx}}\label{sec:proofapprox}
We recall that for any finite sequence $(a_i)_{i\in I},$ and $0<p_1,p_2<\infty,$
\begin{equation*}\label{eq:lqnorms}
\left(\sum_{i\in I} |a_i|^{p_2}\right)^{1/p_2} \leq |I|^{(1/p_2-1/p_1)_+}\left(\sum_{i\in I} |a_i|^{p_1}\right)^{1/p_1}.
\end{equation*}
Besides, we have proved in the course of the proof of Proposition~\ref{prop:combinatorial} that 
$$\bm J_\ell \leq c_1(M,d)(\ell-dj_0+d-1)^{d-1}.$$
In the hyperbolic basis, $f$ admits a unique decomposition of the form 
$$f=\sum_{\ell=dj_0}^\infty\sum_{\bm\lambda\in U\bm{\nabla}(\ell)} \langle f,\Psi_{\bm\lambda}\rangle \Psi^\star_{\bm\lambda}.$$
Defining 
$$f_\bullet=\sum_{\ell=dj_0}^{L_\bullet}\sum_{\bm\lambda\in U\bm{\nabla}(\ell)} \langle f,\Psi_{\bm\lambda}\rangle \Psi^\star_{\bm\lambda},$$
we have for finite $q>0,$ using the aforementioned reminders,
\begin{align*}
\|f-f_\bullet\|^2_{\bm\Psi} 
&= \sum_{\ell=L_\bullet+1}^\infty\sum_{\bm j\in \bm J_\ell}\sum_{\bm\lambda\in \bm{\nabla_j}} \langle f,\Psi_{\bm\lambda}\rangle ^2\\
&\leq \sum_{\ell=L_\bullet+1}^\infty  \sum_{\bm j\in \bm J_\ell} \left(\sharp \bm{\nabla_j}\right)^{2(1/2-1/p)_+}\left( \sum_{\bm\lambda\in\bm{\nabla_j}} |\langle f,\Psi_{\bm\lambda}\rangle|^p\right)^{2/p}\\
&\leq C( B,d,p) \sum_{\ell=L_\bullet+1}^\infty 2^{2\ell (1/2-1/p)_+} \sum_{\bm j\in \bm J_\ell} \left( \sum_{\bm\lambda\in\bm{\nabla_j}} |\langle f,\Psi_{\bm\lambda}\rangle|^p\right)^{2/p}\\
&\leq C(B,d,p) \sum_{\ell=L_\bullet+1}^\infty 2^{2\ell (1/2-1/p)_+} \sharp\bm J_\ell ^{2(1/2-1/q)_+}\left(\sum_{\bm j\in \bm J_\ell} \left( \sum_{\bm\lambda\in\bm{\nabla_j}} |\langle f,\Psi_{\bm\lambda}\rangle|^p\right)^{q/p}\right)^{2/q}\\
&\leq C(B,d,p) \sum_{\ell=L_\bullet+1}^\infty 2^{2\ell (1/2-1/p)_+} (\ell-dj_0+d-1)^{2(d-1)(1/2-1/q)_+}R^22^{-2\ell(\alpha+1/2-1/p)}\\
&\leq C(B,d,p) R^2\sum_{\ell=L_\bullet+1}^\infty (\ell-dj_0+d-1)^{2(d-1)(1/2-1/q)_+}2^{-2\ell(\alpha-(1/p-1/2)_+)}\\
&\leq C(B,\alpha,p,d)R^2 L_\bullet^{2(d-1)(1/2-1/q)_+}  2^{-2L_\bullet(\alpha-(1/p-1/2)_+)}.
\end{align*}
The case $q=\infty$ can be treated in the same way.

Let us fix $k\in\{0,\ldots,L_\bullet-\ell_1\}$ and define $\bar m (\ell_1+k,f)$ as the subset of $U\bm\nabla(\ell_1+k)$ such that $\left\{|\langle f,\Psi_{\bm{\lambda}}\rangle|; \bm\lambda \in \bar m (\ell_1+k,f)\right\}$ are the  $N(\ell_1,k)$  largest elements among $\left\{|\langle f,\Psi_{\bm{\lambda}}\rangle|; \bm\lambda \in U\bm\nabla(\ell_1+k)\right\}$. We then consider the approximation for $f$ given by
$$A(\ell_1,f)=\sum_{\ell=dj_0}^{\ell_1-1}\sum_{\bm\lambda\in U\bm{\nabla}(\ell)} \langle f,\Psi_{\bm\lambda}\rangle \Psi^\star_{\bm\lambda}+
\sum_{k=0}^{L_\bullet-\ell_1}\sum_{\bm\lambda\in\bar m (\ell_1+k,f)} \langle f,\Psi_{\bm\lambda}\rangle^2
$$
and the set 
$$m_{\ell_1}(f)=\left(\bigcup_{\ell=dj_0}^{\ell_1-1} U\bm{\nabla}( \ell) \right)
\cup \left( \bigcup_{k=0}^{L_\bullet-\ell_1} \bar m(\ell_1+k,f)\right).$$ 
Let us first assume that $0<p\leq 2.$ Using Lemma 4.16 in~\cite{Massart} and~\eqref{eq:lqnorms}, we get
\begin{align*}
\|f_\bullet-A(\ell_1,f))\|_{\bm\Psi}^2
&= \sum_{k=0}^{L_\bullet-\ell_1}\sum_{\bm\lambda\in U\bm{\nabla}(\ell_1+k) \backslash \bar m (\ell_1+k,f)} \langle f,\Psi_{\bm\lambda}\rangle ^2\\
&\leq\sum_{k=0}^{L_\bullet-\ell_1}\left(\sum_{\bm\lambda\in U\bm{\nabla}(\ell_1+k)} |\langle f,\Psi_{\bm\lambda}\rangle|^p\right)^{2/p}/(N(\ell_1,k)+1)^{2(1/p-1/2)}\\
&\leq\sum_{k=0}^{L_\bullet-\ell_1}\sharp \bm J_{\ell_1+k} ^{2(1/p-1/q)+}\left(\sum_{\bm j\in \bm J_{\ell_1+k}}\left(\sum_{\bm\lambda\in \bm{\nabla_j}} |\langle f,\Psi_{\bm\lambda}\rangle|^p\right)^{q/p}\right)^{2/q}/(N(\ell_1,k)+1)^{2(1/p-1/2)}.
\end{align*}
Besides, it follows from~\eqref{eq:N} that 
$$ N(\ell_1,k)+1 \geq 2 M^{-d}2^{-d}(d-1)^{-(d-1)} (\ell_1+k-dj_0+d-1)^{d-1} 2^{\ell_1} (k+2)^{-(d+2)}.$$
Therefore 
\begin{equation*}
\|f_\bullet-A(\ell_1,f))\|_{\bm\Psi}^2
\leq C(\alpha,p,d)R^2  (\ell_1-dj_0+d-1)^{2(d-1)(1/2-1/\max(p,q))}2^{-2\alpha \ell_1} .
\end{equation*}
In case $p\geq 2,$ the same kind of upper-bound follows from 
$$\|f_\bullet-A(\ell_1,f))\|_{\bm\Psi}^2
\leq
\sum_{k=0}^{L_\bullet-\ell_1}\sharp U\bm\nabla(\ell_1+k)^{2(1/2-1/p)} \left(\sum_{\bm\lambda\in U\bm{\nabla}(\ell_1+k)} |\langle f,\Psi_{\bm\lambda}\rangle|^p\right)^{2/p}.$$

Last, 
$$ \|f-A(\ell_1,f))\|_{\bm\Psi}^2=\|f-f_\bullet\|_{\bm\Psi}^2+\|f_\bullet-A(\ell_1,f))\|_{\bm\Psi}^2$$
which completes the proof.

\bibliographystyle{alpha}
\bibliography{WaveletTree}

\end{document}